\documentclass[11pt,reqno]{amsart}

\title[]{On the Space-Time Analyticity of the Keller-Segel-Navier-Stokes system}

\author{Elie Abdo}
\address{Department of Mathematics, University of California Santa Barbara, CA 93106, USA.}
\email{elieabdo@ucsb.edu}

\author{Zhongtian Hu}
\address{Department of Mathematics, Duke University, Durham, NC 27708, USA}
\email{zhongtian.hu@duke.edu}

\usepackage[margin=1in]{geometry}
\usepackage{amsmath, amsthm, amssymb}
\usepackage{times}
\usepackage{comment}
\usepackage{color}
\usepackage{hyperref}
\newcommand{\pa}{\partial}
\newcommand{\p}{\partial}
\newcommand{\la}{\label}
\newcommand{\fr}{\frac}
\newcommand{\na}{\nabla}
\newcommand{\be}{\begin{equation*}}
\newcommand{\ee}{\end{equation*}}
\newcommand{\ba}{\begin{array}{l}}
\newcommand{\ea}{\end{array}}
\newcommand{\eps}{\epsilon}

\newtheorem{prop}{Proposition}[section]

\newtheorem{cor}{Corollary}[section]
\newtheorem{rem}{Remark}[section]

\newcommand{\beg}{\begin}

\newtheorem{lem}{Lemma}[section]

\usepackage{MnSymbol,wasysym}

\newcommand{\R}{\mathbb R}

\def\RR{{\mathbb R}}
\def\TT{{\mathbb T}}
\def\NN{\mathbb N}

\def\calM{{\mathcal M}}
\def\phivv{\phi^v_{0,T}}
\def\phiv{\phi^v_{1,T}}
\newcommand{\n}[2]{\frac{({#1})^{#2}}{({#1})!}}

\begin{document}
\begin{abstract}
In this paper, we study the coupled Keller-Segel-Navier-Stokes system, which models chemotaxis occuring in ambient viscous fluid. We consider this nonlinear, nonlocal system on a periodic strip, equipped with homogeneous Neumann boundary conditions for the Keller-Segel part and no-slip boundary condition for the fluid part. We prove the simultaneous space-time analyticity of the solution up to the boundary based on energy methods.
\end{abstract} 


\maketitle
\section{Introduction}

We consider the Keller-Segel equation under the effect of ambient fluid flow modeled by buoyancy-driven Navier-Stokes equation in a periodic strip $\Omega := \TT \times [0,\pi]$:
\begin{equation}\label{eq:ksns}
    \begin{cases}
        \pa_t \rho + u \cdot \na \rho - \Delta \rho + \na \cdot (\rho \na c) = 0,\\
        -\Delta c = \rho - \bar{\rho},\quad \bar\rho = |\Omega|^{-1}\int_\Omega \rho dx,\\
        \pa_t u + u \cdot \na u - \Delta u + \na p = (0, \rho)^T,\\
        \na \cdot u = 0.
    \end{cases}
\end{equation}
We equip the system with initial data $(\rho_0, u_0)$, where $\rho_0$ is nonnegative, as well as homogeneous Neumann boundary conditions
\begin{equation} \la{5}
\na c \cdot n = 0,\quad\quad \na \rho \cdot n = 0,
\end{equation} where $n$ denotes the outward unit normal vector to $\pa \Omega$. 
The first equation in system \eqref{eq:ksns} describes the evolution of a scalar cell/bacteria density $\rho$, secreting an external chemical with density $c$ that attracts the bacteria themselves. In the meantime, we assume that the production and re-arrangement of chemoattractants occur at a much smaller timescale than that of micro-organism's evolution. This assumption allows us to model the chemical secretion via a Poisson equation. Moreover, in realistic settings, it is biologically meaningful to consider chemotaxis occuring in ambient fluids \cite{Pur,Qiu}. This motivates us to consider the coupling of Keller-Segel equation with a Navier-Stokes flow with velocity $u = (u_1,u_2)$ and a pressure $p$. In this work, we consider a specific scenario in which chemotaxis and fluid advection interact via buoyancy force, which is manifested in the forcing term of the Navier-Stokes equation in \eqref{eq:ksns}.

When the fluid flow is absent, i.e. $u \equiv 0$, the system \eqref{eq:ksns} is reduced to the classical parabolic-elliptic Keller-Segel equation, namely
\begin{equation}\label{eq:ks}
    \p_t \rho - \Delta \rho + \nabla\cdot (\rho \na c) = 0,\quad\quad -\Delta c = \rho - \bar\rho,
\end{equation}
equipping homogeneous Neumann boundary conditions \eqref{5}. This model was first proposed by Patlak \cite{Pat}, and Keller and Segel \cite{KS}. Since the proposition of \eqref{eq:ks}, there is a plethora of research dedicated to the well-posedness, singularity formation, and asymptotic behavior of this equation. We refer the readers to the following works \cite{Bed1,BCM,BDP,CGMN1,CGMN2,Hor,HY,Jag,Nag,V1,V2,Wei} and the references therein.

Moreover, there have been many attempts to investigate both regularity properties and quantitative behaviors of solutions for Keller-Segel equation coupling with various fluid models. We refer interested readers to the following works \cite{CKL,DiLM,Win1,Win2} which discuss regularity properties and quantitative behaviors of solutions to such systems. We highlight the fact that the chemotaxis-fluid coupling greatly complicates the analysis: more precisely, the presence of fluid might affect chemotaxis in a highly nontrivial fashion when comparing with the case without fluid advection. For example, global existence of solutions can be achieved near certain nontrivial steady states \cite{DuLM}; sufficiently strong passive flows that possess mixing properties \cite{BH,He1,HT,IXZ,KX} and some active flows that are buoyancy driven \cite{H,HK,HKY} or dissipation-enhancing \cite{He2,ZZZ} are also able to ensure global existence of solutions derived from initial data which can lead to finite-time blowup in the non-advected case. On the other hand, some chemotaxis-fluid models might still form blowup in finite time, possibly in a rather complicated way, see \cite{GH}.

In this paper, we are interested in the analyticity of solutions to the model \eqref{eq:ksns}.
The analyticity of semilinear parabolic models has been extensively studied in the literature under different geometric spatial settings 
and based on different approaches. Fourier series techniques, adapted to Hilbert spaces, have been employed to address the Gevrey regularity of partial differential equations equipped with periodic boundary conditions for Sobolev initial data (see for instance \cite{ ALW, FTS, FT, LG} and references therein). 
Under weaker regularity initial assumptions (namely $L^p$), the spatial analyticity of the Navier-Stokes equations was obtained in \cite{GK} based on a mild formulation of the corresponding complexified problem. A similar approach was adapted to other nonlocal nonlinear models with supernonlinear forcing (see for instance \cite{AI} and references therein). In the presence of physical boundaries, the aforementioned techniques break down, giving rise to the need for novel ideas. In \cite{kv}, the authors established a derivative reduction proof, based on classical energy methods, to study the simultaneous space-time analyticity of the heat and Stokes equations on half-space with homogeneous Dirichlet boundary conditions, and their proof was modified recently in \cite{AW} and generalized to a wider class of boundary conditions (including Neumann and homogeneous Robin). In this paper, we apply those results and successfully tackle the analyticity of the Keller-Segel-Navier-Stokes system, described by \eqref{eq:ksns}, where the challenges do not arise only from the boundary effects but also from the nonlinear and nonlocal aspects of the involved equations.

\subsection{Functional Setting.}  For $1 \le p \le \infty$, we denote by $L^p(\Omega)$ the Lebesgue spaces of measurable functions $f$ from $\Omega$ to $\R$ (or $\RR^2)$ such that 
\be 
\|f\|_{L^p} = \left(\int_{\Omega} \|f\|^p dx\right)^{1/p} <\infty
\ee if $p \in [1, \infty)$ and 
\be 
\|f\|_{L^{\infty}} = {\mathrm{esssup}}_{\Omega}  |f| < \infty
\ee if $p = \infty$. The $L^2$ inner product is denoted by $(\cdot,\cdot)_{L^2}$. 

For $k \in \NN$, we denote by $H^k(\Omega)$ and $\dot{H}^k(\Omega)$ the classical and homogeneous Sobolev spaces of measurable functions $f$ from $\Omega$ to $\R$ (or $\RR^2)$ with weak derivatives of order $k$ such that  
\be 
\|f\|_{H^k}^2 = \sum\limits_{|\alpha| \le k} \|D^{\alpha}f\|_{L^2}^2 < \infty,
\ee and 
\be 
\|f\|_{\dot{H^k}}^2 = \sum\limits_{|\alpha| = k} \|D^{\alpha}f\|_{L^2}^2 < \infty
\ee respectively. 

Since we also work with Navier-Stokes equation, we introduce the space
$$
V := \{v \in (H^1_0(\Omega))^2 \;|\; \nabla\cdot v = 0\},
$$
where the space $H_{0}^{1}(\Omega)$ refers to the subspace of $H^1(\Omega)$ consisting of functions with homogeneous Dirichlet boundary conditions. 

For a Banach space $(X, \|\cdot\|_{X})$ and $p\in [1,\infty]$, we consider the Lebesgue spaces $ L^p(0,T; X)$ of functions $f$  from $X$ to $\R$ (or $\RR^2)$ satisfying 
\be 
\int_{0}^{T} \|f\|_{X}^p dt  <\infty
\ee with the usual convention when $p = \infty$. The corresponding norm will be denoted by $\|\cdot\|_{L^p(0, T; X)}$ or abbreviated as $\|\cdot\|_{L_t^p X}$. 

For a real number $p \in [1,\infty]$, we consider the space-time Lebesgue spaces $L^p((0,T) \times \Omega)$, equipped with the natural norm. The corresponding norms will be denoted by $\|\cdot\|_{L^p((0,T) \times \Omega)}$ or abbreviated as $\|\cdot\|_{L^p_{x,t}}$.

\subsection{Main Result} We consider a regularity exponent $r \ge 3$, a positive time $T$, and strictly positive small quantities $0 < \tilde{\epsilon} \le \bar{\epsilon} \le \epsilon \le 1$. For a smooth vector $\omega = (\rho, u)$, we define the Gevrey type norm 
\begin{equation}
\beg{aligned} \la{G1}
\Phi_T (\omega) = \sum\limits_{i + j + k > r}  \fr{(i+j+k)^r}{(i+j+k)!} \epsilon^i \tilde{\epsilon}^j \bar{\epsilon}^{k} \|t^{i+j+k-r}\pa_t^i \pa_2^j \pa_1^k \omega\|_{L^2((0,T) \times \Omega)}
+ \sum\limits_{i + j + k \le r} \|\pa_t^i \pa_2^j \pa_1^k \omega\|_{L^2((0,T) \times \Omega)}
\end{aligned} 
\end{equation} where $\pa_1$ stands for the tangential (horizontal) derivative and $\pa_2$ stands for the normal (vertical) derivative. All indices over which the sums are taken are assumed to be nonnegative integers.

Under suitable Sobolev regularity assumptions imposed on the initial data, there exists a local-in-time solution to the model \eqref{eq:ksns} that is analytic in both space and time:

\beg{Thm} \la{maintheorem} Let $r \ge 3$ be an integer. There exists $\epsilon, \tilde{\epsilon}, \bar{\epsilon} \in (0,1]$ depending only on $r$ such that for any initial scalar $\rho_0 \in H^{2r}(\Omega)$ and any initial divergence-free velocity field $u_0 \in H^{2r}(\Omega)$ satisfying the compatibility conditions, there exists a time $T_0$ depending on the $H^{2r}$ norm of the initial data and a solution $\omega:= (\rho, u)$ to the initial boundary value problem \eqref{eq:ksns} on the time interval $(0,T_0)$ such that the estimate  
\be 
\Phi_{T_0}(\omega) \le 1
\ee holds.  
\end{Thm}

\begin{rem}
    Besides the coupling via buoyancy, as studied in this work, the Keller-Segel equation can be coupled with the Navier-Stokes equation in various other ways. For example, another biologically interesting forcing for the Navier-Stokes equation is $\rho \nabla c$. Such forcing term models the friction in between the fluid and cell motions, see e.g. \cite{He2}. We remark that using estimates developed in this work, the alternative forcing term mentioned above can be similarly treated. Thus, we focus on the buoyancy forcing in this paper to reduce tedious technicalities.
\end{rem}

The presence of physical boundaries is a source of technical difficulty due to the absence of Fourier techniques. Further challenges arise from studying the analyticity of the Keller-Segel-Navier-Stokes system due to the nonlinear structure driven by $u \cdot \na u, u \cdot \na \rho$, and $\na \cdot (\rho \na c)$, as well as the nonlocal aspects driven by  $\na \cdot (\rho \na c)$. We exploit the elliptic regularity of the Poisson equation obeyed by $c$ to derive elliptic Gevrey estimates from which we obtain good control of the Gevrey norm of $\nabla c$ by that of $\rho$. This consequently deals with the nonlocality of the model. As for the nonlinearities, we use interpolation and combinatorial inequalities to bound the Gevrey norms of the nonlinear terms by time constant multiples of the Gevrey size of the solution itself, yielding consequently the boundedness of the latter over a short period of time. We believe this is the first result that addresses the simultaneous time-space analyticity of a nonlinear model equipped with Neumann boundary conditions using energy methods. 

This paper is organized as follows. In section \ref{S2}, we sketch an outline of the proof of Theorem \ref{maintheorem}. In section \ref{S3}, we address the local well-posedness of the Keller-Segel-Navier-Stokes model in Sobolev spaces. Section \ref{S4} is dedicated to elliptic Gevrey estimates. Finally, we derive in section \ref{S5} good control of the high-frequency Gevrey norms of the nonlinear terms governing the evolution of the system.

\section{Outline of the proof of Theorem \ref{maintheorem}.} \la{S2}
The proof of Theorem \ref{maintheorem} is divided into several major steps.

{\bf{Step 1. Local Regularity in Lebesgue Spaces.}} We prove the existence of a time $T^*$ depending on the $H^{2r}$ norm of the initial data $\omega_0$ such that the problem \eqref{eq:ksns} has a unique local solution  $\omega = (\rho, u)$ on the time interval $(0, T^*)$ satisfying
\be 
\sup\limits_{0 \le t \le T^*} \|\omega(t)\|_{H^{2r}(\Omega)} \le 2\|\omega_0\|_{H^{2r}(\Omega)}.
\ee We also deduce that $\omega = (\rho, u)$  obeys
\be 
\|u \cdot \na \rho\|_{L^2(0,T^*; H^{2r-2}(\Omega))} + \|\na \cdot (\rho \na c)\|_{L^2(0,T^*; H^{2r-2}(\Omega))}  \le CT^{\fr{1}{2}}\|\omega_0\|_{H^{2r}(\Omega)}^2 
\ee for some positive universal constant $C$ and any $T \in (0, T^*)$. This is addressed in section \ref{S3}.

{\bf{Step 2. Decomposition of the Gevrey Norm.}} Consider $T \le \min(1,T^*).$ We decompose the Gevrey norm $\Phi_T(\omega)$ as the sum 
\be 
\Phi_T(\omega) = \phi_{0,T} + \phi_{1,T}  + \phi_{2,T},
\ee where
\be 
\phi_{0,T} = \sum\limits_{i + j + k \le r} \|\pa_t^i \pa_2^j \pa_1^k \rho \|_{L^2((0,T) \times \Omega)}  + \sum\limits_{i + j + k \le r} \|\pa_t^i \pa_2^j \pa_1^k u\|_{L^2((0,T) \times \Omega)}
\ee 
\be 
\phi_{1,T} = \sum\limits_{i + j + k > r}  \fr{(i+j+k)^r}{(i+j+k)!} \epsilon^i \tilde{\epsilon}^j \bar{\epsilon}^{k} \|t^{i+j+k-r}\pa_t^i \pa_2^j \pa_1^k \rho\|_{L^2((0,T) \times \Omega)},
\ee and
\be 
\phi_{2,T}  = \sum\limits_{i + j + k > r}  \fr{(i+j+k)^r}{(i+j+k)!} \epsilon^i \tilde{\epsilon}^j \bar{\epsilon}^{k} \|t^{i+j+k-r}\pa_t^i \pa_2^j \pa_1^k u\|_{L^2((0,T) \times \Omega)}.
\ee 

{\bf{Step 2.1. Good Control of $\phi_{0,T}$.}} We prove that the following estimate
\be 
\phi_{0,T}  \le CT^{\fr{1}{2}} \|\omega_0\|_{H^{2r}(\Omega)} 
\ee holds for any time $T \in (0, T^*)$, where $C$ is a positive universal constant. 

{\bf{Step 2.2. Good Control of $\phi_{1,T}$.}} In the spirit of \cite{AW}, if we assert the following conditions on parameters $\eps, \tilde{\eps}, \bar{\eps}$:
\begin{equation}
    \label{neumanncond}
    \epsilon \le C_{KS},\quad \frac{\bar\epsilon^2}{\epsilon} + \frac{\bar\epsilon}{\epsilon^{1/2}} + \bar\epsilon \le C_{KS},\quad \frac{\tilde\epsilon^2 + \tilde\epsilon + \bar\epsilon\tilde\epsilon}{\epsilon} + \frac{\tilde\epsilon^2}{\epsilon^2} + \frac{\tilde\epsilon}{\epsilon^{1/2}} + \tilde\epsilon^2 + \tilde\epsilon \le C_{KS},\quad \frac{\tilde\epsilon^2}{\bar\epsilon^2} = \frac{1}{4}
\end{equation}
where $C_{KS}$ is a fixed constant only depending on $r$, then the following estimate holds:
\be 
\begin{aligned}
\phi_{1,T} 
&\le C\phi_{0,T}
+ C\|u \cdot \na \rho\|_{L^2(0,T; H^{{2r-2}}(\Omega))}
+ C\|\na \cdot (\rho \na c)\|_{L^2(0,T; H^{{2r-2}}(\Omega))}
\\&+ M_1 + M_2 + M_3 + M_4 + M_5 + M_6,
\end{aligned}
\ee where
\be 
M_1 = \sum\limits_{i + j + k \ge r-2} \fr{(i+j+k+1)^{r-1} \epsilon^i \tilde{\epsilon}^{j+1}\bar{\epsilon}^{k}}{(i+j+k+1)!} \|t^{i+j+k + 2-r} \pa_t^i \pa_2^{j} \pa_1^k (u \cdot \na \rho)  \|_{L^2((0,T) \times \Omega)},
\ee 
\be 
M_2 = \sum\limits_{i \ge r-1} \fr{(i+1)^{r-1} \epsilon^{i+1}}{(i+1)!} \|t^{i+2-r}\pa_t^i \pa_2 (u \cdot \na \rho)\|_{L^2((0,T) \times \Omega)},
\ee
\be 
M_3 = \sum\limits_{i \ge r} \fr{(i+1)^r \epsilon^{i+1} }{(i+1)!} \|t^{i+1-r} \pa_t^i (u \cdot \na \rho)\|_{L^2((0,T) \times \Omega)},
\ee
\be 
M_4 =  \sum\limits_{i + j + k \ge r-2} \fr{(i+j+k+1)^{r-1} \epsilon^i \tilde{\epsilon}^{j+1}\bar{\epsilon}^{k}}{(i+j+k+1)!} \|t^{i+j+k + 2-r} \pa_t^i \pa_2^{j} \pa_1^k \na \cdot (\rho \na c)  \|_{L^2((0,T) \times \Omega)},
\ee
\be 
M_5 = \sum\limits_{i \ge r-1} \fr{(i+1)^{r-1} \epsilon^{i+1}}{(i+1)!} \|t^{i+2-r}\pa_t^i \pa_2 \na \cdot (\rho  \na c)\|_{L^2((0,T) \times \Omega)},
\ee and
\be 
M_6 = \sum\limits_{i \ge r} \fr{(i+1)^r \epsilon^{i+1} }{(i+1)!} \|t^{i+1-r} \pa_t^i \na \cdot (\rho \na c)\|_{L^2((0,T) \times \Omega)}.
\ee We will show in Proposition \ref{prop:M} that 
\be 
\sum\limits_{i=1}^{6} M_i \le C(\rho_0, u_0, r)\left(T + T^{1/2}(\phi_{1,T} + \phi_{2,T}) + T^{r-5/2}\phi_{1,T}(\phi_{1,T} + \phi_{2,T})\right)
\ee for any sufficiently small time $T$. 

{\bf{Step 2.3. Good Control of $\phi_{2,T}$.}} In the spirit of \cite{kv}, if we assert that
\begin{equation}
    \label{dirichletcond}
    \epsilon + \bar\epsilon + \tilde\epsilon^{1/2} + \frac{\tilde\epsilon}{\epsilon} \le C_{NS},
\end{equation}
where $C_{NS}$ is a fixed constant only depending on $r$, then the following estimate holds:
\be 
\begin{aligned}
\phi_{2,T} 
&\le C\phi_{0,T}
+ C\|\rho\|_{L^2(0,T; H^{{2r-2}}(\Omega))}
+ C\|u \cdot \na u\|_{L^2(0,T; H^{{2r-2}}(\Omega))}
\\&+ N_1 + N_2 + N_3 + N_4 + N_5 + N_6 
\end{aligned}
\ee where 
\be 
N_1 = \sum\limits_{i+j+k \ge r-2} \fr{(i+j+k+2)^r \epsilon^i \tilde{\epsilon}^{j+2} \bar{\epsilon}^k}{(i+j+k+2)!} \|t^{i+j+k+2-r} \pa_t^i \pa_2^j \pa_1^k (u \cdot \na u)\|_{L^2((0,T) \times \Omega)} 
\ee 
\be 
N_2 = \sum\limits_{i+k \ge r-2} \fr{(i+k+2)^r \epsilon^i \bar{\epsilon}^{k+2}}{(i+k+2)!} \|t^{i+k+2-r} \pa_t^i \pa_1^k (u \cdot \na u)\|_{L^2((0,T) \times \Omega)} 
\ee 
\be 
N_3 =  \sum\limits_{i\ge r-1} \fr{(i+1)^r \epsilon^{i+1} }{(i+1)!} \|t^{i+1-r} \pa_t^i  (u \cdot \na u)\|_{L^2((0,T) \times \Omega)} 
\ee 
\be 
N_4 = \sum\limits_{i+j+k \ge r-2} \fr{(i+j+k+2)^r \epsilon^i \tilde{\epsilon}^{j+2} \bar{\epsilon}^k}{(i+j+k+2)!} \|t^{i+j+k+2-r} \pa_t^i \pa_2^j \pa_1^k  \rho\|_{L^2((0,T) \times \Omega)} 
\ee 
\be 
N_5 = \sum\limits_{i+k \ge r-2} \fr{(i+k+2)^r \epsilon^i \bar{\epsilon}^{k+2}}{(i+k+2)!} \|t^{i+k+2-r} \pa_t^i \pa_1^k \rho \|_{L^2((0,T) \times \Omega)} 
\ee 
\be 
N_6 = \sum\limits_{i \ge r-1} \fr{(i+1)^r \epsilon^{i+1}}{(i+1)!} \|t^{i+1-r} \pa_t^i  \rho\|_{L^2((0,T) \times \Omega)} 
\ee 
We will show in Proposition \ref{prop:N} that 
\be 
\sum\limits_{i=1}^{6} N_i \le C\left(\phi_{0,T}^{\fr{1}{2}} \Phi_T^{\fr{1}{2}} + T^{\fr{1}{2}} \Phi_T^2 \right)
+ C T\phi_{1,T} + C \phi_{0,T} 
\ee for any sufficiently small time $T$.
\begin{rem}
    From now on, we by default assert that parameters $\epsilon,\bar\eps,\tilde\eps$ satisfy the \textbf{standing assumption}, namely \eqref{neumanncond} and \eqref{dirichletcond}. In fact, it is not hard for one to construct a triplet $(\epsilon, \bar\epsilon, \tilde\eps)$ which satisfies the standing assumption given $r, C_{KS}, C_{NS}$. For example, we may let $C_0 = \min(C_{KS}, C_{NS}, 1)$, and set $\epsilon = C_0/2$, $\bar\epsilon = C_0^2/100$, $\tilde\eps = C_0^2/200$.
\end{rem}

\begin{rem}
    We remark that in Step 2.2 and Step 2.3 above, the choice of constants $C_{KS}$, $C_{NS}$ originates from the linear theory developed in \cite{AW} and \cite{kv} respectively. These two works deal with the half-space $\R^2 \times \R^+$, which does not quite match our setting. However, the only domain-dependent arguments in those works are Calder\'on-Zygmund type elliptic estimates, which still hold in our setting. Therefore, the results in \cite{AW,kv} apply to the domain $\Omega = \TT \times [0,\pi]$, and this justifies our choice of constants $C_{KS}$, $C_{NS}$.
\end{rem}

{\bf{Step 3. Conclusion.}} There exists a small time $T_0$ depending only on the size of the initial data in $H^{2r}(\Omega)$ such that 
\be 
\Phi_T(\omega) \le 1
\ee for all $T \in (0,T_0).$ Indeed, for any $T \in (0,1)$, it follows from the previous steps that 
\begin{equation} \la{bn}
\Phi_T \le C_0 \sqrt{T} (\Phi_T + \Phi_{T}^2 + 1)
\end{equation} for some constant $C_0$ depending on the $H^{2r}(\Omega)$ regularity of the initial data. The relation \eqref{bn} yields the bound
$$\Phi_T \le 1 $$ for any $T \le (1/3C_0)^2$. If not, then there exists a time $T^* \in (0, (1/3C_0)^2]$ such that for any $T < T^*$, it holds that $\Phi_T < 1$ and $\Phi_{T^*} = 1$. Using \eqref{bn} at time $T^*$, we infer that 
$1 < 3C_0 \sqrt{T*}  \le 1$ and obtain a contradiction. 

We point out that all the estimates in this paper are not quite rigorous. In fact, one can truncate the infinite sums in $\Phi_T$ (which we denote by $\Phi_T^N$), observe that $\Phi_T^N$ are well-defined due to parabolic regularity, perform the uniform-in-$N$ Gevrey computations on $\Phi_T^N$, deduce that $\Phi_T^N$ are uniformly bounded by 1 for all times $T \in (0, (1/(3C_0)^2]$, and finally let $N \rightarrow \infty$. We disregard this technicality for the sake of simplicity. 

\section{Local Regularity in Lebesgue Spaces} \la{S3}
In this section, we remark briefly on the local well-posedness for \eqref{eq:ksns}--\eqref{5} in the scale of Sobolev spaces and establish the boundedness of the Sobolev part of $\Phi_T(\omega)$, namely $\phi_{0,T}$. Let us fix $r \ge 3$. For the rest of this section, we assume that the initial data $\rho_0 \in H^{2r}(\Omega)$, $u_0 \in H^{2r}(\Omega) \cap V$ satisfy appropriate compatibility conditions (see e.g. Section 2.4, Chapter 5 in \cite{BF}).  Using a similar approach to that in \cite{HK}, one can show the existence of a time 
$$
T^* = T^*(\|\rho_0\|_{H^{2r}(\Omega)}) > 0
$$
and a unique solution $(\rho, u)$ satisfying the following regularity properties:
$$
\rho \in L^\infty(0,T^*;H^1(\Omega))\cap L^2(0,T^*;H^2(\Omega)),
$$
$$ u \in L^\infty(0,T^*;V)\cap L^2(0,T^*;H^2(\Omega) \cap V),
$$
$$
\p_t \rho, \p_t u \in L^2(0,T^*;L^2(\Omega)),
$$
and
$$
\rho, u \in C^\infty((0,T^*] \times \Omega).
$$ 

In particular, one can prove the following quantitative estimates for higher-order regularity:

\begin{prop}
\label{prop:Sobolev}
Assume that $(\rho,u)$ is a solution to problem \eqref{eq:ksns} with initial data $\rho_0 \in H^{2r}(\Omega)$, $u_0 \in H^{2r}(\Omega) \cap V$ satisfying compatibility conditions. Then the following bounds hold:
\begin{align}
    \sup_{t \in [0,T^*]}\left(\|\p_t^l\rho(t,\cdot)\|_{H^{2r-2l}(\Omega)}^2 + \|\p_t^l u(t,\cdot)\|_{H^{2r-2l}(\Omega)}^2\right) &\le C(\|\rho_0\|_{H^{2r}(\Omega)},\|u_0\|_{H^{2r}(\Omega)},r),\label{est:regest1}
\end{align}
for $0 \le l \le r.$
\end{prop}
\begin{rem}
    We point out several modifications that one has to make when adapting the proof in \cite{HK} to prove the local well-posedness of \eqref{eq:ksns}--\eqref{5} and Proposition \ref{prop:Sobolev}. While still using Galerkin approximation to construct a solution, we need to use the Neumann basis to build approximated solutions to $\rho$. Moreover, time differentiation $\p_t$ still keeps the zero Neumann boundary condition intact i.e. $\p_t \rho$ still satisfies homogeneous Neumann boundary conditions. This means that the approach in \cite{HK} of testing time derivatives in energy estimates carries over to our case. Finally, we need to treat related terms generated by the extra advection term $u\cdot\nabla u$ as we study Navier-Stokes equation in our case. The treatment of such terms, however, are standard in the study of local well-posedness of 2D Navier-Stokes equations with forcing. We refer the interested readers to the proof of Theorem V.2.10. in \cite{BF} for related details.
\end{rem}

Moreover, a bound for both $\phi_{0,T}$ and chemical gradient $\nabla c$ emerges as a direct corollary of Proposition \ref{prop:Sobolev} as follows:

\begin{cor}
    With all assumptions in Proposition \ref{prop:Sobolev}, we have the following estimate:
    \begin{equation}
        \label{est:phi0}
        \sum_{i+j+k \le r}\|\p_t^i\p_2^j\p_1^k\nabla c\|_{L^2_{x,t}} + \phi_{0,T} \le C(\|\rho_0\|_{H^{2r}(\Omega)},\|u_0\|_{H^{2r}(\Omega)}, r)T^{1/2},
    \end{equation}
    for any $T \in (0, T^*)$.
\end{cor}
\begin{proof}
    To show \eqref{est:phi0}, it suffices to consider the highest order terms, namely the space-time $L^2$ norms of 
    $
    \p_t^i\p_2^j\p_1^k \rho$ and $\p_t^i\p_2^j\p_1^k u
    $
    for $i + j + k = r$. Observe that
    \begin{align*}
        \|\p_t^i\p_2^j\p_1^k \rho\|_{L^2_{x,t}([0,T] \times \Omega)}^2 &= \int_{0}^{T}\|\p_t^i\p_2^j\p_1^k \rho(t)\|_{L^2(\Omega)}^2 dt \le \int_{0}^{T}\|\p_t^i \rho(t)\|_{H^{j+k}(\Omega)}^2 dt\\
        &\le \int_{0}^{T}\|\p_t^i \rho(t)\|_{H^{2(j+k)}(\Omega)}^2 dt \le \int_{0}^{T}\|\p_t^i \rho(t)\|_{H^{2r - 2i}(\Omega)}^2 dt\\
        &\le C(\|\rho_0\|_{H^{2r}(\Omega)},\|u_0\|_{H^{2r}(\Omega)}, r)T,
    \end{align*}
    where we used \eqref{est:regest1} in the final inequality. A similar argument also yields $$
    \|\p_t^i\p_2^j\p_1^k u\|_{L^2_{x,t}([0,T] \times \Omega)}^2 \le C(\|\rho_0\|_{H^{2r}(\Omega)},\|u_0\|_{H^{2r}(\Omega)}, r)T.
    $$
    We have thus shown the desired bound for $\phi_{0,T}$. To bound norms corresponding to $\nabla c$, we recall that $c$ verifies the following Neumann problem:
        $$
        -\Delta c = \rho - \bar\rho,\quad \p_2 c|_{\p\Omega} = 0.
        $$
        Taking $\p_t^i\p_2$ on both sides of the Poisson equation, we note that $\p_t^i\p_2 c$ satisfies a Poisson equation equipped with zero Dirichlet boundary condition:
        $$
        -\Delta(\p_t^i\p_2 c) = \p_t^i\p_2 \rho,\quad \p_t^i\p_2 c|_{\p\Omega} = 0.
        $$
        By standard elliptic estimates, we obtain
        \begin{align*}
            \|\p_t^i\p_2^j\p_1^k \p_2 c\|_{L^2(\Omega)} &\le \|\p_t^i\p_2 c\|_{H^{j+k}(\Omega)} \le C\|\p_t^i\p_2 \rho\|_{H^{j+k-2}(\Omega)}\\
            &\le C\|\p_t^i \rho\|_{H^{j+k-1}(\Omega)}.
        \end{align*}
        Then a similar argument as how we treat $\phi_{0,T}$ would yield
        $$
        \|\p_t^i\p_2^j\p_1^k \p_2 c\|_{L^2_{x,t}([0,T] \times \Omega)}^2 \le C(\|\rho_0\|_{H^{2r}(\Omega)},\|u_0\|_{H^{2r}(\Omega)}, r)T.
        $$
        Similarly, we observe that $\p_t^i\p_1 c$ verifies the Neumann problem
        $$
        -\Delta(\p_t^i\p_1 c) = \p_t^i\p_1 \rho,\quad \p_t^i\p_1\p_2 c|_{\p\Omega} = 0,
        $$
        where $\int_\Omega \p_t^i\p_1 \rho dx = 0$. Thus, a similar argument to the above yields
        $$
        \|\p_t^i\p_2^j\p_1^k \p_1 c\|_{L^2_{x,t}([0,T] \times \Omega)}^2 \le C(\|\rho_0\|_{H^{2r}(\Omega)},\|u_0\|_{H^{2r}(\Omega)}, r)T,
        $$
        and the proof is complete.
        
\end{proof}

Finally, we show Sobolev estimates of the following nonlinear terms:
\begin{lem}
    Let $(\rho,u)$ be the solution to \eqref{eq:ksns} on $(0, T^*)$. Then the following estimates hold:
    \begin{equation}
    \begin{aligned}
        \label{est:lowSob}
        &\|u\cdot \nabla \rho\|_{L^2(0,T; H^{2r-2}(\Omega))}+ \|u\cdot \nabla u\|_{L^2(0,T; H^{2r-2}(\Omega))}+ \|\nabla \cdot (\rho\nabla c)\|_{L^2(0,T; H^{2r-2}(\Omega))} 
        \\&\quad\quad\le C(\|\rho_0\|_{L^1(\Omega)},\|\rho_0\|_{H^{2r}(\Omega)},\|u_0\|_{H^{2r}(\Omega)})T^{1/2}
        \end{aligned}
    \end{equation}
    for any $T \in (0, T^*)$.
\end{lem}
\begin{proof}
    Since $2r - 2 \ge 2 > 1$ as $r \ge 3$, we may use the algebra property of $H^{2r-2}$ to treat the first term on the LHS by:
    \begin{align*}
        \|u\cdot \nabla \rho\|_{L^2(0,T; H^{2r-2}(\Omega))}^2 &\le \int_{0}^T \|u\|_{H^{2r-2}(\Omega)}^2\|\rho\|_{H^{2r-1}(\Omega)}^2 dt\\
        &\le \left(\sup_{0 \le t \le T^*}\|\rho\|_{H^{2r}(\Omega)}^2 \right)\left(\sup_{0 \le t \le T^*}\|u\|_{H^{2r}(\Omega)}^2\right)T\\
        &\le C(\|\rho_0\|_{H^{2r}(\Omega)},\|u_0\|_{H^{2r}(\Omega)})T,
    \end{align*}
    where we used \eqref{est:regest1} in the final inequality. The term $\|u\cdot \nabla u\|_{L^2(0,T; H^{2r-2}(\Omega))}$ can be treated in a similar way. To estimate the third term, we observe that
    \begin{align*}
        \|\nabla \cdot (\rho\nabla c)\|_{L^2(0,T; H^{2r-2}(\Omega))}^2 &\le \int_{0}^T\|\rho\nabla c\|_{H^{2r-1}(\Omega)}^2 dt\\
        &\le \int_{0}^T\|\rho\|_{H^{2r-1}(\Omega)}^2\|\nabla c\|_{H^{2r-1}(\Omega)}^2 dt.
    \end{align*}
    By standard elliptic estimates for the Neumann problem, we notice that
    $$
    \|\nabla c\|_{H^{2r-1}(\Omega)} \le \|c\|_{H^{2r}(\Omega)} \le C \|\rho - \rho_M\|_{H^{2r-2}(\Omega)} \le C(\|\rho_0\|_{L^1(\Omega)})\|\rho\|_{H^{2r-2}(\Omega)},
    $$
    where we also used the fact that $\rho_M$ is conserved. Thus, we can continue to estimate the third term by:
    \begin{align*}
        \|\nabla \cdot (\rho\nabla c)\|_{L^2(0,T; H^{2r-2}(\Omega))}^2 &\le C(\|\rho_0\|_{L^1(\Omega)})\left(\sup_{0 \le t \le T^*}\|\rho\|_{H^{2r}(\Omega)}\right)^4(T - 0) \\&\le C(\|\rho_0\|_{L^1(\Omega)},\|\rho_0\|_{H^{2r}(\Omega)},\|u_0\|_{H^{2r}(\Omega)})T.
    \end{align*}
    The proof is completed after we collect all estimates above.
    \end{proof}
\beg{rem} In this work, we assume that the compatibility conditions are satisfied for the sake of simplicity. If we only want to show space-time analyticity in $[t_0,T^*] \times \Omega$ for arbitrary $t_0 > 0$, we may relax the regularity of initial data to be $\rho_0 \in H^1(\Omega), u_0 \in H^1_0 \cap V$. In fact, due to parabolic smoothing, we may evolve \eqref{eq:ksns} to time $t_0/2$ and apply Theorem \ref{maintheorem} with intial data $\rho(t_0/2, x), u(t_0/2,x)$. 
\end{rem}

\section{Gevrey Elliptic Regularity} \la{S4}
In this section, we address the Gevrey regularity of solutions to the Poisson equation relating chemical density $c$ and cell density $\rho$, namely:
\begin{equation}
    \label{poisson}
    -\Delta c = \rho - \bar\rho,
\end{equation}
equipped with homogeneous Neumann boundary conditions.

\beg{prop} \label{ellipticgevreycontrol} Let $\epsilon, \tilde{\epsilon},  \bar{\epsilon}$ satisfy the standing assumption, then there exists a positive constant $C$ depending only on $r$ such that  
\begin{equation} \label{est:gevreyellip}
\beg{aligned}
&\sum\limits_{i + j + k > r} \frac{(i+j+k)^r}{(i+j+k)!} \epsilon^i \tilde{\epsilon}^j \bar{\epsilon}^k \|t^{i+j+k-r} \pa_t^i \pa_2^j \pa_1^k \na c\|_{L^2_{x,t}}
\\&\quad\quad\quad\quad\le C\left(1 + T\tilde{\epsilon} + \fr{T\tilde{\epsilon}^2}{\bar{\epsilon}} \right)\sum\limits_{i+j+k > r} \frac{(i+j+k)^r}{(i+j+k)!} \epsilon^i \tilde{\epsilon}^j \bar{\epsilon}^k \|t^{i+j+k-r} \pa_t^i \pa_2^j \pa_1^k \rho\|_{L^2_{x,t}}
\end{aligned}
\end{equation} holds. 
\end{prop}

\begin{proof} Let $\mathcal{I}$ be the sum
\be 
\mathcal{I} = \sum\limits_{i + j + k > r} \frac{(i+j+k)^r}{(i+j+k)!} \epsilon^i \tilde{\epsilon}^j \bar{\epsilon}^k \|t^{i+j+k-r} \pa_t^i \pa_2^j \pa_1^k \na c\|_{L_{x,t}^2}.
\ee We bound $\mathcal{I}$ by the sum $\mathcal{I}_1 + \mathcal{I}_2$ where
\be 
\mathcal{I}_1 = \sum\limits_{i + j + k > r} \frac{(i+j+k)^r}{(i+j+k)!} \epsilon^i \tilde{\epsilon}^j \bar{\epsilon}^k \|t^{i+j+k-r} \pa_t^i \pa_2^{j+1} \pa_1^k  c\|_{L_{x,t}^2}
\ee and 
\be 
\mathcal{I}_2 = \sum\limits_{i + j + k > r} \frac{(i+j+k)^r}{(i+j+k)!} \epsilon^i \tilde{\epsilon}^j \bar{\epsilon}^k \|t^{i+j+k-r} \pa_t^i \pa_2^{j} \pa_1^{k+1} c\|_{L_{x,t}^2}.
\ee In order to estimate $\mathcal{I}_1$, we decompose it into 
\be 
\mathcal{I}_1 = \mathcal{I}_{1,1} + \mathcal{I}_{1,2}
\ee where
\be 
\mathcal{I}_{1,1} 
= \sum\limits_{i + j + k > r, j \ge 1} \frac{(i+j+k)^r}{(i+j+k)!} \epsilon^i \tilde{\epsilon}^j \bar{\epsilon}^k \|t^{i+j+k-r} \pa_t^i \pa_2^{j+1} \pa_1^k c\|_{L_{x,t}^2}
\ee and
\be 
\mathcal{I}_{1,2} 
= \sum\limits_{i + k > r} \frac{(i+k)^r}{(i+k)!} \epsilon^i  \bar{\epsilon}^k \|t^{i+k-r} \pa_t^i \pa_1^k  \pa_2 c\|_{L_{x,t}^2}.
\ee 
By making use of the Poisson equation obeyed by $c$, we obtain the relation
\begin{equation} \label{pderel}
-\pa_2^2 c = \rho - \bar{\rho} + \pa_1^2 c
\end{equation} and use it to bound 
\be 
\beg{aligned}
\mathcal{I}_{1,1} 
&\le \sum\limits_{i + j + k > r, j \ge 1} \frac{(i+j+k)^r}{(i+j+k)!} \epsilon^i \tilde{\epsilon}^j \bar{\epsilon}^k \|t^{i+j+k-r} \pa_t^i \pa_2^{j-1} \pa_1^{k+2} c\|_{L_{x,t}^2} 
\\&\quad\quad+ \sum\limits_{i + j + k > r, j \ge 1} \frac{(i+j+k)^r}{(i+j+k)!} \epsilon^i \tilde{\epsilon}^j \bar{\epsilon}^k \|t^{i+j+k-r} \pa_t^i \pa_2^{j-1} \pa_1^k \rho\|_{L_{x,t}^2}
\end{aligned}
\ee via an application of the triangle inequality. A re-indexing $k \mapsto k+2, j \mapsto j-2$ yields 
\be 
\beg{aligned}
\mathcal{I}_{1,1} 
&\le \sum\limits_{i + j + k > r} \frac{(i+j+k)^r}{(i+j+k)!} \epsilon^i \tilde{\epsilon}^{j+2} \bar{\epsilon}^{k-2} \|t^{i+j+k-r} \pa_t^i \pa_2^{j+1} \pa_1^{k} c\|_{L_{x,t}^2} 
\\&\quad\quad+ T\sum\limits_{i + j + k > r} \frac{(i+j +k+1 )^r}{(i+j+ k+1)!} \epsilon^i \tilde{\epsilon}^{j+1} \bar{\epsilon}^k \|t^{i+j+k-r} \pa_t^i \pa_2^{j} \pa_1^k \rho\|_{L_{x,t}^2}
\\&= \fr{\tilde{\epsilon}^2}{\bar{\epsilon}^2}\sum\limits_{i + j + k > r} \frac{(i+j+k)^r}{(i+j+k)!} \epsilon^i \tilde{\epsilon}^{j} \bar{\epsilon}^{k} \|t^{i+j+k-r} \pa_t^i \pa_2^{j+1} \pa_1^{k} c\|_{L_{x,t}^2} 
\\&\quad\quad+ 2^{r-1} T\tilde{\epsilon} \sum\limits_{i + j + k > r} \frac{(i+j +k)^r}{(i+j+ k)!} \epsilon^i \tilde{\epsilon}^{j} \bar{\epsilon}^k \|t^{i+j+k-r} \pa_t^i \pa_2^{j} \pa_1^k \rho\|_{L_{x,t}^2}.
\end{aligned}
\ee 
To estimate $\mathcal{I}_{1,2}$, we apply $t^{i+k - r} \pa_t^i \pa_1^k$ on both sides of \eqref{poisson} to obtain the following Neumann boundary value problem
\be 
-\Delta (t^{i+k-r}\pa_t^i \pa_1^k c) = t^{i+k-r}\pa_t^i \pa_1^k \rho,
\ee 
\be 
t^{i+k-r}\pa_t^i \pa_1^k \pa_2 c|_{\pa \Omega} = 0. 
\ee Since $i + k > r > 0$, we must have either $i > 0$ or $k > 0$. In either case, we must have
$$
\int_\Omega t^{i+k-r}\pa_t^i \pa_1^k \rho dx = 0.
$$

Evoking elliptic estimates, we obtain the spatial Sobolev $H^1$ bound
\be 
\|t^{i+k-r}\pa_t^i \pa_1^k c\|_{H^1(\Omega)} \le C\|t^{i+k-r}\pa_t^i \pa_1^k \rho\|_{L^2(\Omega)},
\ee from which we infer that
\be 
\mathcal{I}_{1,2} \le C\sum\limits_{i+k > r} \fr{(i+k)^r}{(i+k)!} \epsilon^i \bar{\epsilon}^k \|t^{i+k - r} \pa_t^i \pa_1^k  \rho \|_{L^2(\Omega)}.
\ee Therefore, the term $\mathcal{I}_1$ can be controlled by
\be 
\beg{aligned}
\mathcal{I}_1 
\le \fr{\tilde{\epsilon}^2}{\bar{\epsilon}^2}\mathcal{I}_1 
+ C(1 + T\tilde{\epsilon}) \sum\limits_{i + j + k > r} \frac{(i+j +k)^r}{(i+j+ k)!} \epsilon^i \tilde{\epsilon}^{j} \bar{\epsilon}^k \|t^{i+j+k-r} \pa_t^i \pa_2^{j} \pa_1^k \rho\|_{L_{x,t}^2}.
\end{aligned}
\ee 
Since $\eps, \bar\eps, \tilde\eps$ satisfy the standing assumption, we particularly have $\fr{\tilde{\epsilon}^2}{\bar{\epsilon}^2} \le \frac{1}{2}$. We hence deduce that
\be 
\mathcal{I}_1 \le C(1 + T\tilde{\epsilon}) \sum\limits_{i + j + k > r} \frac{(i+j +k)^r}{(i+j+ k)!} \epsilon^i \tilde{\epsilon}^{j} \bar{\epsilon}^k \|t^{i+j+k-r} \pa_t^i \pa_2^{j} \pa_1^k \rho\|_{L_{x,t}^2}.
\ee As for $\mathcal{I}_2$, we decompose it into 
\be 
\mathcal{I}_{2} = \mathcal{I}_{2,1} + \mathcal{I}_{2,2} + \mathcal{I}_{2,3},
\ee where 
\be 
\mathcal{I}_{2,1} = \sum\limits_{i + j + k > r, j \ge 2} \frac{(i+j+k)^r}{(i+j+k)!} \epsilon^i \tilde{\epsilon}^j \bar{\epsilon}^k \|t^{i+j+k-r} \pa_t^i \pa_2^{j} \pa_1^{k+1} c\|_{L_{x,t}^2},
\ee
\be 
\mathcal{I}_{2,2} = \sum\limits_{i + 1 + k > r} \frac{(i+1+k)^r}{(i+1+k)!} \epsilon^i \tilde{\epsilon} \bar{\epsilon}^k \|t^{i+1+k-r} \pa_t^i \pa_2 \pa_1^{k+1} c\|_{L_{x,t}^2},
\ee and
\be 
\mathcal{I}_{2,3} = \sum\limits_{i + k > r} \frac{(i+k)^r}{(i+k)!} \epsilon^i  \bar{\epsilon}^k \|t^{i+k-r} \pa_t^i \pa_1^{k+1} c\|_{L_{x,t}^2}.
\ee Using again the relation \eqref{pderel}, we estimate $\mathcal{I}_{2,1}$ as $\mathcal{I}_{1,1}$ and obtain
\be
\beg{aligned}
\mathcal{I}_{2,1} 
&\le \fr{\tilde{\epsilon}^2}{\bar{\epsilon}^2}\sum\limits_{i + j + k > r} \frac{(i+j+k)^r}{(i+j+k)!} \epsilon^i \tilde{\epsilon}^{j} \bar{\epsilon}^{k} \|t^{i+j+k-r} \pa_t^i \pa_2^{j} \pa_1^{k+1} c\|_{L_{x,t}^2} 
\\&\quad\quad+ \fr{ 2^{r-1} T\tilde{\epsilon}^2}{\bar{\epsilon}} \sum\limits_{i + j + k > r} \frac{(i+j +k)^r}{(i+j+ k)!} \epsilon^i \tilde{\epsilon}^{j} \bar{\epsilon}^k \|t^{i+j+k-r} \pa_t^i \pa_2^{j} \pa_1^k \rho\|_{L_{x,t}^2}.
\end{aligned}
\ee 
To control $\mathcal{I}_{2,2}$, we note that since $t^{i+1+k-r} \pa_t^i \pa_1^{k+1} \pa_2 c$ solves the Dirichlet problem
\be 
-\Delta (t^{i+1+k-r} \pa_t^i \pa_1^{k+1} \pa_2 c) = t^{i+1+k-r} \pa_t^i \pa_1^{k+1} \pa_2 \rho,
\ee 
\be 
t^{i+1+k-r} \pa_t^i \pa_1^{k+1} \pa_2 c|_{\pa \Omega} = 0,
\ee we have the $H^1$ elliptic estimate
\be 
\|t^{i+1+k-r} \pa_t^i \pa_1^{k+1} \pa_2 c\|_{L^2(\Omega)} \le \|t^{i+1+k-r} \pa_t^i \pa_1^{k+1} \pa_2 c\|_{H^1(\Omega)} \le C\|t^{i+1+k-r} \pa_t^i \pa_1^{k} \pa_2 \rho \|_{L^2(\Omega)}.
\ee As a consequence, we deduce the boundedness of $\mathcal{I}_{2,2}$ by the Gevrey norm of $\rho$ and obtain
\be 
\mathcal{I}_{2,2} 
\le C\sum\limits_{i+1+k > r} \frac{(i+1+k)^r}{(i+1+k)!} \epsilon^i \bar{\epsilon}^k \| t^{i+1+k-r} \pa_t^i \pa_2  \pa_1^{k} \rho\|_{L_{x,t}^2}.
\ee
To control $\mathcal{I}_{2,3}$, we note that $t^{i+k-r} \pa_t^i \pa_1^{k+1}c$ solves the Neumann boundary value problem
\be 
-\Delta (t^{i+k-r} \pa_t^i \pa_1^{k+1}c) = t^{i+k-r} \pa_t^i \pa_1^{k+1}\rho,
\ee 
\be 
\pa_2 (t^{i+k-r} \pa_t^i \pa_1^{k+1}c)|_{\pa \Omega} = 0
\ee with $t^{i+k-r} \pa_t^i \pa_1^{k+1}\rho$ obeying
\be 
\int_{\Omega} t^{i+k-r} \pa_t^i \pa_1^{k+1}\rho dx = 0.
\ee By the elliptic regularity of this problem, we infer that 
\be 
\|t^{i+k-r} \pa_t^i \pa_1^{k+1}c\|_{H^1(\Omega)} \le C\|t^{i+k-r} \pa_t^i \pa_1^{k}\rho\|_{L^2(\Omega)}.
\ee
Thus,
\be 
\mathcal{I}_{2,3} \le C\sum\limits_{i+k > r} \frac{(i+k)^r}{(i+k)!} \epsilon^i \bar{\epsilon}^k \|t^{i+k - r} \pa_t^i \pa_1^k \rho\|_{L_{x,t}^2}.
\ee As the ratio $\frac{\tilde{\epsilon}^2}{\bar{\epsilon}^2}$ is chosen to be less than $\fr{1}{2}$, we conclude that
\be 
\mathcal{I}_2 \le C\left(1 + \fr{T\tilde{\epsilon}^2}{\bar{\epsilon}} \right)\sum\limits_{i+j+k > r} \frac{(i+j+k)^r}{(i+j+k)!} \epsilon^i \tilde{\epsilon}^j \bar{\epsilon}^k \|t^{i+j+k-r} \pa_t^i \pa_2^j \pa_1^k \rho\|_{L_{x,t}^2}.
\ee 
This completes the proof of Proposition \ref{ellipticgevreycontrol}. 
\end{proof}

\section{Gevrey Boundedness of Solutions to the Keller-Segel-Navier-Stokes System} \la{S5}
In this section, we seek good control of the Gevrey sums $\phi_{1,T}$ and $\phi_{2,T}$. This amounts to prove suitable estimates for terms $M_i$ and $N_i$.


\subsection{Interpolation Inequalities}
We first remark on several Gagliardo-Nirenberg-Sobolev type inequalities that will be repeatedly used for the rest of this section. Let $v \in H^2(\Omega)$, the following inequalities are classical:
    \begin{align*}
        \|v\|_{L^\infty(\Omega)} &\le C \|v\|_{\dot{H}^2(\Omega)}^{1/2}\|v\|_{L^2(\Omega)}^{1/2} + \|v\|_{L^2(\Omega)},\\
        \|v\|_{L^4(\Omega)} &\le C \|v\|_{\dot{H}^1(\Omega)}^{1/2}\|v\|_{L^2(\Omega)}^{1/2} + \|v\|_{L^2(\Omega)}.
    \end{align*}
If $v \in H^1(0,T; H^2(\Omega))$, we also have the following space-time analogs:
\begin{align}
    \|v\|_{L^\infty_{x,t}} &\le C \|\p_t v\|_{L^2_t\dot{H}^2_x}^{1/2}\|v\|_{L^2_t\dot{H}^2_x}^{1/2} + \|v\|_{L^2_t\dot{H}^2_x} + \|\p_t v\|_{L^2_{x,t}} + \|v\|_{L^2_{x,t}},\label{est:Linfty}\\
    \|v\|_{L^\infty_tL^4_x} &\le C \|\p_t v\|_{L^2_t\dot{H}^1_x}^{1/2}\|v\|_{L^2_t\dot{H}^1_x}^{1/2} + \|v\|_{L^2_t\dot{H}^1_x} + \|\p_t v\|_{L^2_{x,t}} + \|v\|_{L^2_{x,t}}. \label{est:L4}
\end{align}
With the help of the inequalities above, one may derive the following estimates of terms that are weighted by time:
\begin{align}
    \|t^{i+j+k}\p_t^i\p_2^j\p_1^k v\|_{L^\infty_{x,t}} &\le C \left(\|t^{i+j+k}\p_t^{i+1}\p_2^j\p_1^k v\|_{L^2_t\dot{H}^2_x}^{1/2} + (i+j+k)^{1/2}\|t^{i+j+k-1}\p_t^{i}\p_2^j\p_1^k v\|_{L^2_t\dot{H}^2_x}^{1/2}\right) \notag\\
    &\times \|t^{i+j+k}\p_t^{i}\p_2^j\p_1^k v\|_{L^2_t\dot{H}^2_x}^{1/2} + \|t^{i+j+k}\p_t^{i+1}\p_2^j\p_1^k v\|_{L^2_{x,t}} \notag\\
    &+ (i+j+k)\|t^{i+j+k-1}\p_t^{i}\p_2^j\p_1^k v\|_{L^2_{x,t}} \label{est:weightedLinfty}\\
    \|t^{i+j+k-1}\p_t^i\p_2^j\p_1^k v\|_{L^\infty_tL^4_x} &\le C \left(\|t^{i+j+k-1}\p_t^{i+1}\p_2^j\p_1^k v\|_{L^2_t\dot{H}^1_x}^{1/2} + (i+j+k)^{1/2}\|t^{i+j+k-2}\p_t^{i}\p_2^j\p_1^k v\|_{L^2_t\dot{H}^1_x}^{1/2}\right) \notag\\
    &\times \|t^{i+j+k-1}\p_t^{i}\p_2^j\p_1^k v\|_{L^2_t\dot{H}^1_x}^{1/2} + \|t^{i+j+k-1}\p_t^{i+1}\p_2^j\p_1^k v\|_{L^2_{x,t}} \notag\\
    &+ (i+j+k)\|t^{i+j+k-2}\p_t^{i}\p_2^j\p_1^k v\|_{L^2_{x,t}}, \label{est:weightedL4}
\end{align}
where both \eqref{est:weightedLinfty} and \eqref{est:weightedL4} hold for $t \le 1$ and $i+j+k \ge 1$. We refer the derivations of the inequalities above to Section 4.1 of \cite{CKV}.

\subsection{Combinatorial Inequalities} The following combinatorial fact will be used repeatedly in the upcoming subsections. If $i, j, m, n$ are nonnegative integers such that $j \le i$ and $n \le m$, then it holds that
\begin{equation} \la{comb1}
{i+m \choose j+n} \ge {i \choose j} {m \choose n}
\end{equation} Indeed, this follows from the Vandermonde's identity 
$${i+m \choose j+n} = \sum\limits_{k=0}^{j+n} {i \choose k} {m \choose j+n-k},$$ where the latter sum is obviously bounded from below by the term corresponding to $k = j$, yielding \eqref{comb1}.
As a consequence, we deduce that 
\begin{equation} \label{est:comb1}
{i \choose \ell} {j \choose n}{k \choose m} \frac{(\ell + m + n)!}{(i+j+k)!} (i+j+k - l - n - m)! \le 1
\end{equation} for any nonnegative integers $i, \ell, j, n, k, m$ obeying $\ell \le i, n \le j, m \le k$. 

\subsection{Estimates of $M_i$}
This subsection is dedicated to the proof of the following proposition, which treats the nonlinear estimates in the Keller-Segel part of the system. 
\begin{prop}\label{prop:M}
    There is a positive constant $C$ only depending on initial data $\rho_0, u_0$ and parameter $r$ such that
    \begin{equation}\label{est:M}
    \sum_{i = 1}^6 M_i \le C(\rho_0, u_0, r)\left(T + T^{1/2}(\phi_{1,T} + \phi_{2,T}) + T^{r-5/2}\phi_{1,T}(\phi_{1,T} + \phi_{2,T})\right).
    \end{equation}
\end{prop}

To start, we observe several simplifications. We may first rewrite \eqref{eq:ksns} as follows:
$$
\p_t \rho - \Delta \rho + \nabla\cdot(\rho (u + \nabla c)) = 0.
$$
The divergence structure of the nonlinear terms above motivates us to consider the following more general terms:
\begin{align*}
    \calM_1(v) &= \sum\limits_{i + j + k \ge r-2} \fr{(i+j+k+1)^{r-1} \epsilon^i \tilde{\epsilon}^{j+1}\bar{\epsilon}^{k}}{(i+j+k+1)!} \|t^{i+j+k + 2-r} \pa_t^i \pa_2^{j} \pa_1^k (\nabla\cdot(\rho v))  \|_{L^2((0,T) \times \Omega)},\\
    \calM_2(v) &= \sum\limits_{i \ge r-1} \fr{(i+1)^{r-1} \epsilon^{i+1}}{(i+1)!} \|t^{i+2-r}\pa_t^i \pa_2 (\nabla\cdot(\rho v))\|_{L^2((0,T) \times \Omega)},\\
    \calM_3(v) &= \sum\limits_{i \ge r} \fr{(i+1)^r \epsilon^{i+1} }{(i+1)!} \|t^{i+1-r} \pa_t^i (\nabla\cdot(\rho v))\|_{L^2((0,T) \times \Omega)},
\end{align*}
where $v = (v_1, v_2)$ denotes an analytic vector field. We set the Sobolev norm of $v$
$$
\phivv := \sum_{i+j+k \le r}\|\p_t^i\p_2^j\p_1^k v\|_{L^2_{x,t}},
$$
and the high-frequency analytic norm
$$
\phiv := \sum\limits_{i + j + k > r} \frac{(i+j+k)^r}{(i+j+k)!} \epsilon^i \tilde{\epsilon}^j \bar{\epsilon}^k \|t^{i+j+k-r} \pa_t^i \pa_2^j \pa_1^k v\|_{L^2_{x,t}}.
$$
It is worth noticing that we can further reduce the control of $\calM_i$, $i = 1,2,3$.
\begin{prop}
    There exists a positive constant $C(r)$ depending only on $r$ such that
    \begin{equation} \la{simply3}
\mathcal{M}_1(v) + \mathcal{M}_2(v) + \mathcal{M}_3(v)
\le C(r)\epsilon^{r-1} \phi_{0,T}\phi_{0,T}^{v}
+ C(r)\left(T + \epsilon \tilde{\epsilon}^{-2} + \epsilon \tilde{\epsilon}^{-1} \right) \tilde{\mathcal{M}}_1(v)
    \end{equation} holds, where
    \be 
\tilde{\mathcal{M}}_1(v) = \sum\limits_{i+j+k \ge r-1} \frac{(i+j+k+1)^{r}\epsilon^i \tilde{\epsilon}^{j+1}\bar{\epsilon}^k}{(i+j+k+1)!} \|t^{i+j+k+1-r} \pa_t^i \pa_2^j \pa_1^k \na \cdot (\rho v)\|_{L^2((0,T) \times \Omega)}.
    \ee 
\end{prop}

\begin{proof}
For the sake of clarity, we suppress the $v$-dependence of the terms $\calM_i$ and $\tilde{\calM}_1$ below. Note that we may decompose $\mathcal{M}_1$ into the sum $\mathcal{M}_{11} + \mathcal{M}_{12}$ where
\be 
\mathcal{M}_{11} =\sum\limits_{i+j+k = r-2} \frac{(i+j+k+1)^{r-1}\epsilon^i \tilde{\epsilon}^{j+1}\bar{\epsilon}^k}{(i+j+k+1)!} \| \pa_t^i \pa_2^j \pa_1^k \na \cdot (\rho v)\|_{L^2((0,T) \times \Omega)}
\ee and 
\be 
\mathcal{M}_{12} =\sum\limits_{i+j+k \ge r-1} \frac{(i+j+k+1)^{r-1}\epsilon^i \tilde{\epsilon}^{j+1}\bar{\epsilon}^k}{(i+j+k+1)!} \|t^{i+j+k+2-r} \pa_t^i \pa_2^j \pa_1^k \na \cdot (\rho v)\|_{L^2((0,T) \times \Omega)}.
\ee Since $r \ge 3$, $H^{r-1}((0,T) \times \Omega)$ is a Banach Algebra, and consequently we can estimate $\mathcal{M}_{11}$ as follows,
\begin{equation} \label{simply1}
\beg{aligned}
\mathcal{M}_{11}
&\le C(r) \epsilon^{r-1} \|\rho v\|_{H^{r-1}((0,T) \times \Omega)}\\
&\le C(r) \epsilon^{r-1} \|\rho\|_{H^{r-1}((0,T) \times \Omega)} \|v\|_{H^{r-1}((0,T) \times \Omega)}\\
&\le C(r) \epsilon^{r-1} \phi_{0,T} \phi_{0,T}^v.
\end{aligned}
\end{equation} As for $\mathcal{M}_{12}$, we have the straightforward bound
\be 
\mathcal{M}_{12}
\le T \tilde{\mathcal{M}}_{1}.
\ee In view of the algebraic inequality 
\be 
\frac{(i+1)^{r-1}}{(i+1)!} \le \frac{(i+2)^{r}}{(i+2)!},
\ee we estimate 
\be 
\mathcal{M}_2 \le \frac{2\epsilon}{\tilde{\epsilon}^2} \sum\limits_{i+1 \ge r} \frac{(i+1+1)^{r}}{(i+1+1)!}\epsilon^i \tilde{\epsilon}^2 \|t^{i+1+1-r} \pa_t \pa_2 \na \cdot (\rho v)\|_{L^2((0,T)\times \Omega)}
\le \frac{2\epsilon}{\tilde{\epsilon}^2} \tilde{\mathcal{M}}_1.
\ee We finally bound $\mathcal{M}_3$ by 
\begin{equation} \label{simply2}
\mathcal{M}_3 \le \frac{\epsilon}{\tilde{\epsilon}} \tilde{\mathcal{M}}_1.
\end{equation} Putting \eqref{simply1}--\eqref{simply2} together, we deduce \eqref{simply3}.
 \end{proof}

Now, we demonstrate the following estimate of $\tilde{\calM}_1(v)$.
\beg{prop} Suppose $\eps,\bar\eps, \tilde\eps$ satisfy the standing assumption and $r \ge 3$. Then we have
\begin{equation} \label{tM1}
\tilde{\calM}_{1}(v) \le C(r)\left(\phi_{0,T}^v \phi_{0,T} + \phi_{1,T}^v \phi_{0,T} + \phi_{0,T}^v \phi_{1,T} + T^{r-5/2}\phi_{1,T}^v \phi_{1,T}\right).
\end{equation}
\end{prop}

\begin{proof}
   For convenience, we suppress the dependence of $\tilde{\calM}_1$ on $v$. We also introduce the following multi-index notation:
   \be 
\pa^{\alpha} = \pa_t^i \pa_2^j \pa_1^k, \hspace{0.5cm} |\alpha| = i + j + k, \hspace{0.5cm} {\alpha \choose \beta} = \frac{\alpha!}{\beta! (\alpha - \beta)!}, \hspace{0.5cm}  \varepsilon^\alpha = \eps^i\tilde{\eps}^j\Bar{\eps}^k
   \ee Then we may rewrite $\tilde{\mathcal M}_1$ as
   \be 
\tilde{\mathcal M}_1 = \sum\limits_{|\alpha| \ge r -1} \frac{(|\alpha|+1)^r\tilde{\eps}\varepsilon^\alpha}{(|\alpha|+1)!} \|t^{|\alpha| + 1 -r} \pa^{\alpha} \na \cdot (\rho v)\|_{L_{x,t}^2}.
   \ee We may further decompose $\calM_1$ into $\tilde{\mathcal M}_{11}$ and $\tilde{\mathcal M}_{12}$, where
   \be 
\tilde{\mathcal M}_{11} = \sum\limits_{|\alpha| \ge r-1} \frac{(|\alpha|+1)^r \tilde{\eps}\varepsilon^\alpha}{(|\alpha| +1)!} \|t^{|\alpha| + 1 -r} \pa^{\alpha} (v \cdot \na \rho)\|_{L_{x,t}^2}
   \ee and 
   \be 
\tilde{\mathcal M}_{12} = \sum\limits_{|\alpha| \ge r-1} \frac{(|\alpha|+1)^r \tilde{\eps}\varepsilon^\alpha}{(|\alpha| +1)!} \|t^{|\alpha| + 1 -r} \pa^{\alpha} (\rho \na \cdot v)\|_{L_{x,t}^2}.
   \ee We start by estimating $\tilde{\mathcal{M}}_{11}$. We have
   \be 
\tilde{\mathcal{M}}_{11} = \tilde{\mathcal{M}}_{111} 
+ \tilde{\mathcal{M}}_{112}
+ \tilde{\mathcal{M}}_{113},
   \ee where
   \be 
\tilde{\mathcal{M}}_{111} = \frac{r^r}{r!} \sum\limits_{|\alpha| = r-1} \tilde{\eps}\varepsilon^\alpha \|\pa^{\alpha} (v \cdot \na \rho)\|_{L_{x,t}^2},
   \ee 
   \be 
\tilde{\mathcal{M}}_{112} = \frac{(r+1)^r}{(r+1)!} \sum\limits_{|\alpha| = r}\tilde{\eps}\varepsilon^\alpha \|t \pa^{\alpha} (v \cdot \na \rho)\|_{L_{x,t}^2},
   \ee and 
   \be 
\tilde{\mathcal{M}}_{113} = \sum\limits_{|\alpha| > r} \frac{(|\alpha|+1)^r}{(|\alpha| +1)!} \tilde{\eps}\varepsilon^\alpha \|t^{|\alpha| + 1 -r} \pa^{\alpha} (v \cdot \na \rho)\|_{L_{x,t}^2}.
   \ee 
   
   \begin{enumerate}
       \item \textbf{Estimate of $\tilde{\calM}_{111}$.} Since $r \ge 3$, $H^{r-1}((0,T) \times \Omega)$ is a Banach Algebra, and thus
   \begin{align} \label{M111}
\tilde{\mathcal{M}}_{111}
&\le C(r)\tilde{\eps}\varepsilon^\alpha \|v \cdot \na \rho\|_{H^{r-1}((0,T) \times \Omega)}\notag\\ 
&\le C(r) \tilde{\eps}\varepsilon^\alpha \|v\|_{H^{r-1}((0,T) \times \Omega)} \|\rho\|_{H^r((0,T) \times \Omega)}\notag\\ 
&\le C(r) \tilde{\eps}\varepsilon^\alpha \phi_{0,T}^v \phi_{0,T}\notag\\
&\le C(r) \eps^r \varepsilon^\alpha \phi_{0,T}^v \phi_{0,T}.
   \end{align} 

    \item \textbf{Estimate of $\tilde{\calM}_{112}$.} By the product rule, we first estimate that
\begin{align*}
    \|t\p^\alpha(v\cdot \nabla \rho)\|_{L^2_{x,t}} &\le \sum_{\beta \le \alpha}{\alpha \choose \beta}\|t\p^\beta v \cdot \nabla \p^{\alpha- \beta}\rho\|_{L^2_{x,t}}\\
    &\le C(r)\Big(\|tv\cdot\nabla \p^\alpha \rho\|_{L^2_{x,t}} + \|t\p^\alpha v \cdot \nabla \rho\|_{L^2_{x,t}} + \sum\limits_{|\beta| = 1}\|t\p^\beta v \cdot \nabla \p^{\alpha - \beta}\rho\|_{L^2_{x,t}}\\
    &\quad\quad+ \sum\limits_{|\beta| = 1}\|t\p^{\alpha - \beta}v \cdot \nabla\p^\beta \rho\|_{L^2_{x,t}} + \sum_{2\le |\beta| \le |\alpha| - 2}\|t\p^\beta v \cdot \nabla \p^{\alpha-\beta}\rho\|_{L^2_{x,t}}\Big)\\
    &\le C(r)\big(\|v\|_{L^\infty_{x,t}}\|t\nabla\p^\alpha\rho\|_{L^2_{x,t}} + \|\p^\alpha v\|_{L^2_{x,t}}\|t\nabla \rho\|_{L^\infty_{x,t}}
        + \sum_{|\beta| = 1}\|\p^\beta v\|_{L^\infty_tL^4_x}\|t\nabla \p^{\alpha - \beta} \rho\|_{L^2_tL^4_x} 
    \\&\quad\quad+ \sum_{|\beta| = |\alpha| - 1}\|\p^{\beta}v\|_{L^2_tL^4_x}\|t\nabla\p^{\alpha -\beta} \rho\|_{L^\infty_tL^4_x}
    + \sum_{2\le |\beta| \le |\alpha| - 2}\|t\p^\beta v\|_{L^\infty_tL^4_x}\|\nabla\p^{\alpha - \beta}\rho\|_{L^2_tL^4_x}\big)\\
    &=: C(r)(\tilde{\calM}_{1121}^\alpha + \tilde{\calM}_{1122}^\alpha + \tilde{\calM}_{1123}^\alpha + \tilde{\calM}_{1124}^\alpha + \tilde{\calM}_{1125}^\alpha)
\end{align*}
    Now, we estimate the terms above separately. First, using \eqref{est:Linfty} and definitions of the Gevrey norms, we have
    \begin{align}
        \sum_{|\alpha| = r}\tilde{\calM}_{1121}^\alpha &\le C\left(\|\p_t v\|_{L^2_t\dot{H}^2_x}^{1/2}\|v\|_{L^2_t\dot{H}^2_x}^{1/2} + \|v\|_{L^2_t\dot{H}^2_x} + \|\p_t v\|_{L^2_{x,t}} + \|v\|_{L^2_{x,t}}\right)\left(\sum_{|\alpha| = r}\sum_{l=1}^2\|t\p_l\p^\alpha\rho\|_{L^2_{x,t}}\right) \notag\\
        &\le C\tilde{\eps}^{-r-1}\phivv\phi_{1,T},
        \label{M1121}
        \end{align}
        Similarly, we can also obtain
        \begin{align}
        \sum_{|\alpha| = r}\tilde{\calM}_{1122}^\alpha &\le \sum_{|\alpha| = r}\|\p^\alpha v\|_{L^2_{x,t}}\left(\|t\p_1 \rho\|_{L^\infty_{x,t}} + \|t\p_2 \rho\|_{L^\infty_{x,t}}\right)\notag
        \\&\le C \phivv\sum_{l = 1}^2 \Big[\left(\|t\p_t\p_l\rho\|_{L^2_t\dot{H}^2_x}^{1/2} + \|\p_l\rho\|_{L^2_t\dot{H}^2_x}^{1/2}\right)\|t\p_l \rho\|_{L^2_t\dot{H}^2_x}^{1/2} + \|t\p_t\p_l\rho\|_{L^2_{x,t}} + \|\p_l\rho\|_{L^2_{x,t}}\Big]\notag
        \\&\le C \phivv\left(\left(\tilde\epsilon^{-2}\phi_{1,T}^{1/2} + \phi_{0,T}^{1/2}\right)T^{1/2}\phi_{0,T}^{1/2} + T\phi_{0,T} + \phi_{0,T}\right)\notag
        \\&\le C \phivv\left(T^{1/2}\tilde\epsilon^{-2}\phi_{1,T}^{1/2}\phi_{0,T}^{1/2} + \phi_{0,T}\right), \label{M1122}
        \end{align}
        Using \eqref{est:L4} and the classical Gagliardo-Nirenberg-Sobolev inequality, we have
        \begin{align}
        \sum_{|\alpha| = r}\tilde{\calM}_{1123}^\alpha
        &\le C \sum_{|\alpha| = r}\sum_{|\beta| = 1}\left(\|\p_t \p^\beta v\|_{L^2_t\dot{H}^1_x}^{1/2}\|\p^\beta v\|_{L^2_t\dot{H}^1_x}^{1/2} + \|\p^\beta v\|_{L^2_t\dot{H}^1_x} + \|\p_t\p^\beta v\|_{L^2_{x,t}} + \|\p^\beta v\|_{L^2_{x,t}}\right)\notag
        \\&\hspace{5cm}\times \left(\|t\nabla \p^{\alpha - \beta} \rho\|_{L^2_{x,t}}^{1/2}\|t\nabla \p^{\alpha - \beta} \rho\|_{L^2\dot{H}^1_x}^{1/2} + \|t\nabla \p^{\alpha - \beta} \rho\|_{L^2_{x,t}}\right)\notag
        \\&\le C\phivv\left(T^{1/2}\tilde\epsilon^{-\frac{r+1}{2}}\phi_{0,T}^{1/2}\phi_{1,T}^{1/2} + T\phi_{0,T}\right), \label{M1123}
        \end{align}
        Similarly, we obtain
     \begin{align}       
        \sum_{|\alpha| = r}\tilde{\calM}_{1124}^\alpha 
        &\le C \sum_{|\alpha| = r}\sum_{|\beta| = |\alpha| - 1}\left(\|\p^\beta v\|_{L^2_t\dot{H}^1_x}^{1/2}\|\p^\beta v\|_{L^2_{x,t}}^{1/2} + \|\p^\beta v\|_{L^2_{x,t}}\right)\bigg[\left(\|t\p_t\p^{\alpha - \beta}\nabla \rho\|_{L^2_t\dot{H}^1_x}^{1/2} + \|\p^{\alpha - \beta} \nabla \rho\|_{L^2_t\dot{H}^1_x}^{1/2}\right)\notag
        \\&\hspace{5cm}\times \|t\p^{\alpha - \beta}\nabla \rho\|_{L^2_t\dot{H}^1_x}^{1/2} + \|t\p_t\p^{\alpha - \beta}\nabla \rho\|_{L^2_{x,t}} + \|\p^{\alpha - \beta}\nabla \rho\|_{L^2_{x,t}}\bigg]\notag
        \\&\le C \phivv\left((\tilde\epsilon^{-2}\phi_{1,T}^{1/2} + \phi_{0,T}^{1/2})T^{1/2}\phi_{0,T}^{1/2} + T\phi_{0,T} + \phi_{0,T}\right),\label{M1124}
        \end{align}
        Finally, using a similar argument to the estimates of the two terms above, we have
        \begin{align}
        \sum_{|\alpha| = r}\tilde{\calM}_{1125}^\alpha  
        &\le C\sum_{|\alpha| = r}\sum_{|\beta| = 2}^{|\alpha|-2} T\left(\|\p_t\p^\beta v\|_{L^2_t\dot{H}^1_x}^{1/2}\|\p^\beta v\|_{L^2_t\dot{H}^1_x}^{1/2} + \|\p^\beta v\|_{L^2_t\dot{H}^1_x} + \|\p_t\p^\beta v\|_{L^2_{x,t}} + \|\p^\beta v\|_{L^2_{x,t}}\right)\notag
        \\&\hspace{5cm}\times\left(\|\nabla \p^{\alpha - \beta}\rho\|_{L^2_t\dot{H}^1_x}^{1/2}\|\nabla \p^{\alpha - \beta}\rho\|_{L^2_{x,t}}^{1/2} + \|\nabla \partial^{\alpha - \beta} \rho\|_{L_{x,t}^2}\right)\notag
        \\&\le CT\phi_{0,T}\phivv. \label{M1125}
        \end{align}
        Combining estimates \eqref{M1121}--\eqref{M1125} and using $T < 1$, we conclude that
        \begin{equation}
            \label{M112}
            \tilde{\calM}_{112} \le C\left(\frac{\eps}{\tilde{\eps}}\right)^r(\phivv\phi_{0,T} + \phivv\phi_{1,T} + \phivv\phi_{0,T}^{1/2}\phi_{1,T}^{1/2})
        \end{equation}

    \item \textbf{Estimate of $\tilde{\calM}_{113}$}
    Now we turn to the estimate of the term $\tilde{\calM}_{113}$. We first split the terms in the following fashion:
    \begin{align*}
        \tilde{\calM}_{113} &\le C\sum_{|\alpha| > r}\n{|\alpha|+1}{r}\tilde{\eps}\varepsilon^\alpha\left(\|v\|_{L^\infty_{x,t}}\|t^{|\alpha|+1 - r}\p^\alpha \nabla \rho\|_{L^2_{x,t}} + \sum_{|\beta| = 1}\|\p^\beta v\|_{L^\infty_tL^4_x}\|t^{|\alpha| +1 - r}\p^{\alpha - \beta}\nabla \rho\|_{L^2_tL^4_x}\right)\\ 
        &\quad+ C\sum_{|\alpha| > r}\n{|\alpha|+1}{r}\tilde{\eps}\varepsilon^\alpha\|\nabla \rho\|_{L^\infty_t L^4_x}\|t^{|\alpha|+1 - r}\p^\alpha v\|_{L^2_tL^4_x}
        \\&\quad\quad+ C\sum_{|\alpha| > r}\sum_{|\beta| = 2}^{\lfloor\frac{|\alpha|}{2}\rfloor}\n{|\alpha|+1}{r}\tilde{\eps}\varepsilon^\alpha{\alpha \choose \beta}\|t^{|\beta|}\p^\beta v\|_{L^\infty_{x,t}}\|t^{|\alpha|+1-|\beta|-r}\p^{\alpha - \beta}\nabla \rho\|_{L^2_{x,t}}
        \\&\quad\quad\quad\quad+ C\sum_{|\alpha| > r}\sum_{|\beta| = \lfloor\frac{|\alpha|}{2}\rfloor+1}^{|\alpha|-1}\n{|\alpha|+1}{r}\tilde{\eps}\varepsilon^\alpha{\alpha \choose \beta}\|t^{|\beta|+1-r}\p^{\beta} v\|_{L^2_tL^4_x}\|t^{|\alpha|-|\beta|}\p^{\alpha - \beta}\nabla \rho\|_{L^\infty_tL^4_x}\\
        &= \tilde{\calM}_{1131} + \tilde{\calM}_{1132} + \tilde{\calM}_{1133} + \tilde{\calM}_{1134}.
    \end{align*}
    To estimate $\tilde{\calM}_{1131}$, we use \eqref{est:Linfty} and \eqref{est:L4} to obtain that
    \begin{align}
        \tilde{\calM}_{1131} &\le C\left((\|\p_t v\|_{L^2_t\dot{H}^2_x}^{1/2}\|v\|_{L^2_t\dot{H}^2_x}^{1/2}) + \|v\|_{L^2_t\dot{H}^2_x} + \|\p_t v\|_{L^2_{x,t}} + \|v\|_{L^2_{x,t}}\right)\phi_{1,T}\notag\\
        &+C\sum_{|\beta| = 1}\left((\|\p_t \p^\beta v\|_{L^2_t\dot{H}^1_x}^{1/2}\|\p^\beta v\|_{L^2_t\dot{H}^1_x}^{1/2}) + \|\p^\beta v\|_{L^2_t\dot{H}^1_x} + \|\p_t \p^\beta v\|_{L^2_{x,t}} + \| \p^\beta v\|_{L^2_{x,t}}\right)\phi_{1,T}\notag\\
        &\le C\phivv\phi_{1,T}. \label{M1131}
    \end{align}
    Using \eqref{est:L4}, we first note that as $r \ge 3$,
    \begin{align*}
    \|\nabla \rho\|_{L^\infty_t L^4_x} &\le C(\|\p_t\nabla\rho\|_{L^2_t\dot{H}^1_x}^{1/2}\|\nabla\rho\|_{L^2_t\dot{H}^1_x}^{1/2} + \|\nabla\rho\|_{L^2_t\dot{H}^1_x} + \|\p_t\nabla\rho\|_{L^2_{x,t}} + \|\nabla\rho\|_{L^2_{x,t}}) \le C\phi_{0,T}.
    \end{align*}
    Thus we may estimate $\tilde{\calM}_{1132}$ by:
    \begin{align}
        \tilde{\calM}_{1132} &\le C\phi_{0,T}\sum_{|\alpha| > r}\n{|\alpha|+1}{r}\tilde{\eps}\varepsilon^\alpha\|t^{|\alpha|+1 - r}\p^\alpha v\|_{L^2_tL^4_x}\notag\\
        &\le C\phi_{0,T}\sum_{|\alpha| > r}\n{|\alpha|+1}{r}\tilde{\eps}\varepsilon^\alpha(\|t^{|\alpha| +1 - r}\p^\alpha v\|_{L^2_t\dot{H}^1_x}^{1/2}\|t^{|\alpha|+1 - r}\p^\alpha v\|_{L^2_{x,t}}^{1/2} + \|t^{|\alpha|+1 - r}\p^\alpha v\|_{L^2_{x,t}})\notag\\
        &\le CT^{1/2}\tilde{\epsilon}^{1/2}\phi_{0,T}\sum_{|\alpha| > r}\n{|\alpha|+1}{r}\left(\frac{(|\alpha|+1)!|\alpha|!}{(|\alpha|+1)^r|\alpha|^r}\right)^{1/2}\left(\tilde{\eps}\varepsilon^\alpha\n{|\alpha|+1}{r}\|t^{|\alpha|+1 - r}\p^\alpha v\|_{L^2_t\dot{H}^1_x}\right)^{1/2}\notag\\
        &\quad\times \left(\varepsilon^\alpha\n{|\alpha|}{r}\|t^{|\alpha| - r}\p^\alpha v\|_{L^2_{x,t}}\right)^{1/2} + CT\tilde{\eps} \phi_{0,T}\sum_{|\alpha| > r}\n{|\alpha|+1}{r}\varepsilon^\alpha\|t^{|\alpha| - r}\p^\alpha v\|_{L^2_{x,t}}\notag\\
        &\le C(r)(T^{1/2} + T)\phi_{0,T}\phiv \le C(r)T^{1/2}\phi_{0,T}\phiv, \label{M1132}
    \end{align}
    where we used $T < 1$, and the fact that $\tilde\eps/\bar\eps$ is a universal constant.
    
    To estimate $\tilde{\calM}_{1133}$, we first apply the combinatorial inequality \eqref{est:comb1} and then the estimate \eqref{est:weightedLinfty} to obtain
    \begin{align*}
        \tilde{\calM}_{1133} &\le C\sum_{|\alpha| > r}\sum_{|\beta| = 2}^{\lfloor\frac{|\alpha|}{2}\rfloor}\n{|\alpha|+1}{r}\tilde{\eps}\varepsilon^\alpha{|\alpha| \choose |\beta|}\|t^{|\beta|}\p^\beta v\|_{L^\infty_{x,t}}\|t^{|\alpha|+1-|\beta|-r}\p^{\alpha - \beta}\nabla \rho\|_{L^2_{x,t}}\\
        &\le C\sum_{|\alpha| > r}\sum_{|\beta| = 2}^{\lfloor\frac{|\alpha|}{2}\rfloor}\frac{(|\alpha|+1)^{r-1}}{|\beta|!(|\alpha|-|\beta|+1)^{r-1}}\varepsilon^{\beta}\|t^{|\beta|}\p^\beta v\|_{L^\infty_{x,t}}\left(\n{|\alpha|+1-|\beta|}{r}\tilde{\eps}\varepsilon^{\alpha-\beta}\|t^{|\alpha|+1-|\beta|-r}\p^{\alpha-\beta}\nabla \rho\|_{L^2_{x,t}}\right)\\
        &\le C\sum_{|\alpha| > r}\sum_{|\beta| = 2}^{\lfloor\frac{|\alpha|}{2}\rfloor}\frac{(|\alpha|+1)^{r-1}}{|\beta|!(|\alpha|-|\beta|+1)^{r-1}}\varepsilon^{\beta}\Big((\|t^{|\beta|}\p_t\p^\beta v\|_{L^2_t\dot{H}^2_x}^{1/2} + |\beta|^{1/2}\|t^{|\beta|-1}\p^\beta v\|_{L^2_t\dot{H}^2_x}^{1/2})\|t^{|\beta|}\p^\beta v\|_{L^2_t\dot{H}^2_x}^{1/2}\\
        &\quad\quad+ \|t^{|\beta|}\p_t\p^\beta v\|_{L^2_{x,t}} + |\beta|\|t^{|\beta|-1}\p^\beta v\|_{L^2_{x,t}}\Big)\left(\n{|\alpha|+1-|\beta|}{r}\tilde{\eps}\varepsilon^{\alpha-\beta}\|t^{|\alpha|+1-|\beta|-r}\p_t^{\alpha-\beta}\nabla \rho\|_{L^2_{x,t}}\right).
    \end{align*}
    We estimate the most singular term above (i.e. the term with most derivatives) and others would follow from a similar argument. Using $r \ge 3$, we may estimate the most singular term as follows:
    \begin{align*}
        C&\sum_{|\alpha| > r}\sum_{|\beta| = 2}^{\lfloor\frac{|\alpha|}{2}\rfloor}\frac{(|\alpha|+1)^{r-1}}{|\beta|!(|\alpha|-|\beta|+1)^{r-1}}\varepsilon^{\beta}\Big((\|t^{|\beta|}\p_t\p^\beta v\|_{L^2_t\dot{H}^2_x}^{1/2}\|t^{|\beta|}\p^\beta v\|_{L^2_t\dot{H}^2_x}^{1/2}\Big)\left(\n{|\alpha|+1-|\beta|}{r}\tilde{\eps}\varepsilon^{\alpha-\beta}\|t^{|\alpha|+1-|\beta|-r}\p_t^{\alpha-\beta}\nabla \rho\|_{L^2_{x,t}}\right)\\
        &\le CT^{r-5/2}\tilde{\eps}^{-5/2}\sum_{|\alpha| > r}\sum_{|\beta| = 2}^{\lfloor\frac{|\alpha|}{2}\rfloor}\left(\n{|\beta|+3}{r}\sum_{\substack{\gamma_1 = 1,\\\gamma_2 + \gamma_3 = 2}}\epsilon^{\beta+\gamma}\|t^{|\beta|+3-r}\p^{\beta + \gamma }v\|_{L^2_{x,t}}\right)^{1/2}\\
        &\times \left(\n{|\beta|+2}{r}\sum_{\substack{\gamma_1 = 0,\\\gamma_2 + \gamma_3 = 2}}\epsilon^{\beta+\gamma}\|t^{|\beta|+2-r}\p^{\beta + \gamma}v\|_{L^2_{x,t}}\right)^{1/2} \times\left(\n{|\alpha|+1-|\beta|}{r}\tilde{\eps}\varepsilon^{\alpha-\beta}\|t^{|\alpha|+1-|\beta|-r}\p_t^{\alpha-\beta}\nabla \rho\|_{L^2_{x,t}}\right)\\
        &\le CT^{r-5/2}\tilde{\eps}^{-5/2}\phiv\phi_{1,T},
    \end{align*}
    where we used $2 \le |\beta| \le \lfloor\frac{|\alpha|}{2}\rfloor$ in the first inequality and Young's convolution inequality in the final inequality. We also used the fact that $\tilde{\eps} \le \bar\eps \le \eps \le 1$. Applying a similar argument to other terms, we may finally estimate $\tilde{\calM}_{1133}$ by
    \begin{align}
        \tilde{\calM}_{1133} &\le C\left(T^{r-5/2}\tilde{\eps}^{-5/2} + T^{r-5/2}\tilde{\eps}^{-2} + T^{r-1}\tilde{\eps}^{-1} + T^{r-1}\right)\phiv\phi_{1,T} + (T + T^2)\phivv\phi_{1,T}\notag\\
        &\le C(T^{r-5/2}\tilde{\eps}^{-5/2}\phiv\phi_{1,T} + T\phivv\phi_{1,T}), \label{M1133}
    \end{align}
    where we used $r \ge 3$ and $T,\tilde{\eps} < 1$.
    
    Finally, to estimate $\tilde{\calM}_{1134}$, we evoke \eqref{est:L4}, \eqref{est:weightedL4} and the combinatorial inequality \eqref{est:comb1} to deduce that
    \begin{align*}
        \tilde{\calM}_{1134} &\le C\sum_{|\alpha| > r}\sum_{|\beta| = \lfloor\frac{|\alpha|}{2}\rfloor+1}^{|\alpha|-1}\n{|\alpha|+1}{r}\tilde{\eps}\varepsilon^{\alpha}{|\alpha| \choose |\beta|}\left(\|t^{|\beta|+1-r}\p^\beta v\|_{L^2_t\dot{H}^1_x}^{1/2}\|t^{|\beta|+1-r}\p^\beta v\|_{L^2_{x,t}}^{1/2} + \|t^{|\beta|+1-r}\p^\beta v\|_{L^2_{x,t}}\right)\\
        &\times \Big((\|t^{|\alpha|-|\beta|}\p_t\p^{\alpha - \beta}\nabla^2 \rho\|_{L^2_{x,t}}^{1/2} + (|\alpha|-|\beta|+1)^{1/2}\|t^{|\alpha|-|\beta|-1}\p^{\alpha-\beta}\nabla^2 \rho\|_{L^2_{x,t}}^{1/2})\|t^{|\alpha|-|\beta|}\p^{\alpha-\beta}\nabla^2 \rho\|_{L^2_{x,t}}^{1/2}\\
        &+ \|t^{|\alpha|-|\beta|}\p_t\p^{\alpha - \beta}\nabla \rho\|_{L^2_{x,t}} + (|\alpha|-|\beta|+1)\|t^{|\alpha|-|\beta|-1}\p^{\alpha-\beta}\nabla \rho\|_{L^2_{x,t}} \Big).
    \end{align*}
    As in estimates for $\tilde{\calM}_{1133}$, we treat the most singular term, while the rest of the terms will follow from a similar argument. The most singular term is given by:
    \begin{align*}
        C&\sum_{|\alpha| > r}\sum_{|\beta| = \lfloor\frac{|\alpha|}{2}\rfloor+1}^{|\alpha|-1}\n{|\alpha|+1}{r}\tilde{\eps}\varepsilon^{\alpha}{|\alpha| \choose |\beta|}\|t^{|\beta|+1-r}\p^\beta v\|_{L^2_t\dot{H}^1_x}^{1/2}\|t^{|\beta|+1-r}\p^\beta v\|_{L^2_{x,t}}^{1/2}\|t^{|\alpha|-|\beta|}\p_t\p^{\alpha - \beta}\nabla^2 \rho\|_{L^2_{x,t}}^{1/2}\|t^{|\alpha|-|\beta|}\p^{\alpha-\beta}\nabla^2 \rho\|_{L^2_{x,t}}^{1/2}\\
        &\le CT^{r-2}\tilde{\eps}^{-2}\sum_{|\alpha| > r}\sum_{|\beta| = \lfloor\frac{|\alpha|}{2}\rfloor+1}^{|\alpha|-1}\left(\n{|\beta|+1}{r}\sum_{\substack{\gamma_1 = 0,\\\gamma_2 + \gamma_3 = 1}}\varepsilon^{\beta+\gamma}\|t^{|\beta|+1-r}\p^{\beta+\gamma} v\|_{L^2_{x,t}}\right)^{1/2}\left(\frac{|\beta|^r}{|\beta|!}\varepsilon^\beta\|t^{|\beta|-r}\p^\beta v\|_{L^2_{x,t}}\right)^{1/2}\\
        &\quad\quad\times \left(\n{|\alpha|-|\beta|+3}{r}\sum_{\substack{\gamma_1 = 1,\\\gamma_2 + \gamma_3 = 2}}\varepsilon^{\alpha-\beta+\gamma}\|t^{|\alpha|-|\beta|+|\gamma|-r} \p^{\alpha - \beta + \gamma} \rho\|_{L^2_{x,t}}\right)^{1/2}\\
        &\quad\quad\times \left(\n{|\alpha|-|\beta|+2}{r}\sum_{\substack{\gamma_1 = 0,\\\gamma_2 + \gamma_3 = 2}}\varepsilon^{\alpha-\beta+\gamma}\|t^{|\alpha|-|\beta|+|\gamma|-r}\p^{\alpha-\beta + \gamma}\rho\|_{L^2_{x,t}}\right)^{1/2}\\
        &\le CT^{r-2}\tilde\eps^{-2}\phiv \phi_{1,T},
    \end{align*}
    where we used $\lfloor\frac{|\alpha|}{2}\rfloor + 1 \le |\beta| \le |\alpha|-1$ in the first inequality and Young's convolution inequality in the second inequality. Proceeding in a similar way, we may conclude the following estimate:
    \begin{align}
    \tilde{\calM}_{1134} &\le C(T^{r-2}\tilde{\eps}^{-2} + T^{r-2}\tilde{\eps}^{-3/2} + T^{r-3/2}\tilde{\eps}^{-3/2}\notag\\
    &+ T^{r-3/2}\tilde{\eps}^{-1/2} + T^{r-1}\tilde{\eps}^{-1} + T^{r-1})\phiv\phi_{1,T} + (T^{r-2} + T^{r-3/2} + T^{r-1})\phiv\phi_{0,T}\notag\\
    &\le C(T^{r-2}\tilde{\eps}^{-2}\phiv\phi_{1,T} + T^{r-2}\phiv\phi_{0,T}), \label{M1134}
    \end{align}
    where we used the fact that $r \ge 3, T < 1, \epsilon < 1$.
    Combining estimates \eqref{M1131}--\eqref{M1134}, we conclude that
    \begin{align}
    \label{M113}
    \tilde{\calM}_{113} \le C(T^{r-5/2}\tilde{\eps}^{-5/2}\phiv\phi_{1,T} + T^{r-2}\phiv\phi_{0,T} + \phivv\phi_{1,T}).
    \end{align}

    \item \textbf{Estimate of $\tilde{\calM}_{11}$}
    Combining estimates \eqref{M111}, \eqref{M112}, \eqref{M113}, we conclude that
    \begin{align}
        \label{M11}
        \tilde{\calM}_{11} &\le C(r) \eps^{r} \phi_{0,T}^v \phi_{0,T} + C\left(\frac{\eps}{\tilde{\eps}}\right)^r(\phivv\phi_{0,T} + \phivv\phi_{1,T} + \phivv\phi_{0,T}^{1/2}\phi_{1,T}^{1/2})\notag\\
        &+ C(r)(T^{r-5/2}\tilde{\eps}^{-5/2}\phiv\phi_{1,T} + T^{r-2}\phiv\phi_{0,T} + \phivv\phi_{1,T})\notag\\
        &\le C(r)\left(\phi_{0,T}^v \phi_{0,T} + \phi_{0,T}^v \phi_{1,T} + T^{r-5/2}\phi_{1,T}^v \phi_{1,T} + T^{r-2}\phi_{1,T}^v \phi_{0,T}\right),
    \end{align}
    where we used that the choices of $\eps, \bar{\eps},\tilde\eps$ only depend on $r$ and particularly is independent of $T$.
    \item {\bf{Estimate of $\tilde{\calM}_{12}$.}} Notice that $\tilde{\calM}_{12}$ has the same structure as $\tilde{\calM}_{11}$. In fact, an estimate of $\tilde{\calM}_{12}$ follows from that of $\tilde{\calM}_{11}$ with the role of $\rho$ and $v$ interchanged:
    \begin{equation}
        \label{M12}
        \tilde{\calM}_{12} \le C(r)\left(\phi_{0,T}^v \phi_{0,T} + \phi_{1,T}^v \phi_{0,T} + T^{r-5/2}\phi_{1,T}^v \phi_{1,T} + T^{r-2}\phi_{0,T}^v \phi_{1,T}\right)
    \end{equation}
   \end{enumerate}
   Finally, we combine estimates \eqref{M11}, \eqref{M12} and the fact that $T^{r-2} < 1$ to obtain:
   $$
   \tilde{\calM}_{1} \le C(r)\left(\phi_{0,T}^v \phi_{0,T} + \phi_{1,T}^v \phi_{0,T} + \phi_{0,T}^v \phi_{1,T} + T^{r-5/2}\phi_{1,T}^v \phi_{1,T}\right).
   $$
   The proof is thus completed.
\end{proof}

Based on the preliminary results above, we are ready to estimate $\sum_{i = 1}^6 M_i$. Combining \eqref{simply3} and \eqref{tM1}, we deduce that given vector field $v$,
\begin{equation}
    \label{est:Maux1}
    \calM_1(v) + \calM_2(v) + \calM_3(v) \le C(r)\left(\phi_{0,T}^v \phi_{0,T} + \phi_{1,T}^v \phi_{0,T} + \phi_{0,T}^v \phi_{1,T} + T^{r-5/2}\phi_{1,T}^v \phi_{1,T}\right),
\end{equation}
where we used that $T < 1$ and that the choice of $\eps, \bar\eps, \tilde\eps$ only depends on $r$. We first observe that by definition
\begin{align*}
    M_i = \calM_i(u),\quad M_{i+3} = \calM_i(\nabla c),
\end{align*}
for $i = 1,2,3$. Moreover, we recall that by definition of $\phi_{0,T}^v$ and estimate \eqref{est:phi0}:
\begin{equation}
    \label{est:Maux2}
\phi_{0,T}^u \le \phi_{0,T} \le C(\rho_0, u_0, r)T^{1/2},\quad \phi_{1,T}^u = \phi_{2,T}.
\end{equation}
By \eqref{est:phi0} and Gevrey elliptic regularity estimate \eqref{est:gevreyellip}, we also have
\begin{equation}
    \label{est:Maux3}
\phi_{0,T}^{\na c} \le C(\rho_0, u_0, r)T^{1/2},\quad \phi_{1,T}^{\na c} \le C\left(1 + T\tilde{\epsilon} + \fr{T\tilde{\epsilon}^2}{\bar{\epsilon}}\right)\phi_{1,T} \le C\phi_{1,T},
\end{equation}
where we used $\tilde\eps < \bar\eps$ and $T < 1$ in the second inequality of the estimate for $\phi_{1,T}^{\na c}$. Finally, we insert \eqref{est:Maux2}, \eqref{est:Maux3} into \eqref{est:Maux1} to conclude \eqref{est:M}. This concludes the proof of Proposition \ref{prop:M}.

\subsection{Estimates of $N_i$}
We treat the nonlinear forcing terms in the Navier-Stokes equation in this section.
\beg{prop}\label{prop:N} Let $T \in (0,1]$. The following estimates
\begin{equation} \label{gevest1}
N_1 + N_2 + N_3 
\le C\left(\phi_{0,T}^{\fr{1}{2}} \phi_T^{\fr{1}{2}} + T^{\fr{1}{2}} \phi_T^2 \right)
\end{equation} and 
\begin{equation} \label{gevest2}
N_4 + N_5 + N_6 
\le C \left(T^2 \tilde{\epsilon}^2 + T^2 \bar{\epsilon}^2 + T\epsilon\right)\phi_{1,T} + C\left(\tilde{\epsilon}^2 + \bar{\epsilon}^2 + \epsilon \right) \phi_{0,T} 
\end{equation} hold for some universal positive constant $C$ depending only on $r$.
\end{prop}

\beg{proof} The bound \eqref{gevest1} follows from \cite{CKV}. As for \eqref{gevest2}, we estimate the terms $N_4$, $N_5$, and $N_6$ separately. 
We start by splitting $N_4$ into the sum 
\be 
N_4 = N_{4,1} + N_{4,2}
\ee 
where
\be 
N_{4,1} = \sum\limits_{i+j+k > r} \fr{(i+j+k+2)^r \epsilon^i \tilde{\epsilon}^{j+2} \bar{\epsilon}^k}{(i+j+k+2)!} \|t^{i+j+k+2-r} \pa_t^i \pa_2^j \pa_1^k  \rho\|_{L^2((0,T) \times \Omega)} 
\ee and 
\be 
N_{4,2} = \sum\limits_{r - 2 \le i+j+k \le r} \fr{(i+j+k+2)^r \epsilon^i \tilde{\epsilon}^{j+2} \bar{\epsilon}^k}{(i+j+k+2)!} \|t^{i+j+k+2-r} \pa_t^i \pa_2^j \pa_1^k  \rho\|_{L^2((0,T) \times \Omega)}. 
\ee In view of the algebraic inequality
\be 
\fr{(i+j+k+2)^r}{(i+j+k+2)!}
\le C(r) \fr{(i+j+k)^r}{(i+j+k)!}
\ee that holds for all positive indices $i, j, k$ obeying $i+ j + k \ge 1$, we bound 
\begin{equation} \label{gevest3}
\beg{aligned}
N_{4,1}
&\le C(r) T^2 \tilde{\epsilon}^2 \sum\limits_{i + j + k > r} \fr{(i+j+k)^r}{(i+j+k)!} \epsilon^i \tilde{\epsilon}^j \bar{\epsilon}^k \|t^{i+j+k-r} \pa_t \pa_2^j \pa_1^k \rho\|_{L^2_{x,t}((0,T) \times \Omega)}
\\&\le C(r) T^2 \tilde{\epsilon}^2 \phi_{1,T}.
\end{aligned}
\end{equation} Due to the smallness of the parameters $\epsilon, \tilde{\epsilon}$, and $\bar{\epsilon}$ and the restriction of the time variable to the interval $(0,T)$ that lies within $[0,1]$, we have
\begin{equation} \la{gevest4}
\beg{aligned}
N_{4,2} \le C(r)\tilde{\epsilon}^2 \sum\limits_{i+j+k \le r} \|\pa_t^i \pa_2^j \pa_1^k \rho\|_{L^2 ((0,T) \times \Omega)} 
\le C(r) \tilde{\epsilon}^2 \phi_{0,T}.
\end{aligned}
\end{equation} Putting \eqref{gevest3} and \eqref{gevest4} together, we infer that
\be 
N_{4} \le C(r) T^2 \tilde{\epsilon}^2 \phi_{1,T} + C(r) \tilde{\epsilon}^2 \phi_{0,T}. 
\ee Similarly, we can bound $N_5$ and $N_6$ by
\be 
N_5 \le C(r) T^2 \bar{\epsilon}^2 \phi_{1,T} + C(r) \bar{\epsilon}^2 \phi_{0,T}
\ee and 
\be 
N_6 \le C(r) T {\epsilon} \phi_{1,T} + C(r) {\epsilon} \phi_{0,T}
\ee respectively. Adding the latter three estimates, we obtain \eqref{gevest2}.

\end{proof}

\textbf{Acknowledgement.} ZH acknowledges partial support from NSF-DMS grant 2306726.
\vspace{0.5cm}

{\bf{Data Availability Statement.}} The research does not have any associated data.

\vspace{0.5cm}

{\bf{Conflict of Interest.}} The authors declare that they have no conflict of interest.

\end{document}